\documentclass[preprint,a4paper, reqno,11pt,ejs,twoside]{imsart}

\usepackage[round,comma]{natbib}
\usepackage{a4wide}
\usepackage{amsmath}
\usepackage{amsthm,amsopn}
\usepackage{amssymb,xcolor}
\usepackage{abstract}

\newtheorem{theorem}{Theorem}
\newtheorem{proposition}{Proposition}
\newtheorem{lemma}{Lemma}
\newtheorem{corollary}{Corollary}

\theoremstyle{definition}
\newtheorem{definition}{Definition}

\theoremstyle{remark}
\newtheorem{remark}{Remark}

\setattribute{thebibliography}{size}{\normalsize}

\def\NL{\mathcal N}
\def\1{\mathbf 1}
\def\RR{\mathbb R}

\def\NN{\mathbb N}
\def\ZZ{\mathbb Z}
\def\ba{\mathbf a}
\def\bq{\mathbf q}
\def\bl{{\boldsymbol l}}
\def\bm{{\boldsymbol m}}
\def\by{\mathbf y}
\def\bx{\mathbf x}

\def\bz{\mathbf z}
\def\bt{\mathbf t}
\def\bc{\mathbf c}

\def\bt{\mathbf t}
\def\bu{\mathbf u}
\def\bv{\mathbf v}
\def\bw{\mathbf w}

\def\bbeta{\boldsymbol\beta}
\def\xxi{\boldsymbol\xi}
\def\oomega{\boldsymbol\omega}

\def\ttheta{\boldsymbol\theta}
\def\ttau{\boldsymbol\tau}
\def\aalpha{{\boldsymbol\alpha}}
\def\ssigma{{\boldsymbol\sigma}}
\def\nnu{\boldsymbol\nu}
\def\ttR{\text{\tt R}}
\def\ttM{\text{\tt M}}
\def\ttT{\text{\tt T}}
\def\ttA{\text{\tt A}}
\def\ttB{\text{\tt B}}
\def\ttI{\text{\tt I}}

\def\mcT{\mathcal T}
\def\mcL{\mathcal L}
\def\mcM{\mathcal M}
\def\mcH{\mathcal H}
\def\mcN{\mathcal N}

\def\mcF{\mathcal F}
\def\mcD{\mathcal D}
\def\mcA{\mathcal A}

\def\ttA{{\tt A}}

\def\bk{{\boldsymbol k}}
\def\bj{{\boldsymbol j}}

\def\pf{\Pi f}
\def\proj{\widehat{\Pi f}_n}

\def\Tr{\mathop{\rm Tr}}
\def\matnorm#1{|\!|\!|#1|\!|\!|}

\begin{document}

\begin{frontmatter}
\title{Minimax testing of a composite null hypothesis defined via a quadratic functional in the model of regression}

\runtitle{Minimax testing of hypotheses defined via quadratic functionals}

\begin{aug}

\author{\fnms{La\"etitia} \snm{Comminges}\ead[label=e1]{laetitia.comminges@enpc.fr}}
\and
\author{\fnms{Arnak S.} \snm{Dalalyan}\ead[label=e2]{dalalyan@imagine.enpc.fr}}
\address{LIGM/IMAGINE and ENSAE/CREST/GENES\\
          \printead{e1,e2}}
\affiliation{Universit\'e Paris Est/ ENPC}

\runauthor{Comminges, L. and Dalalyan, A.}
\end{aug}

\begin{abstract}
We consider the problem of testing a particular type of composite null hypothesis under a nonparametric multivariate regression model.
For a given quadratic functional $Q$, the null hypothesis states that the regression function $f$ satisfies the constraint $Q[f]=0$, while
the alternative corresponds to the functions for which $Q[f]$ is bounded away from zero.
On the one hand, we provide minimax rates of testing and the exact separation constants, along with a sharp-optimal testing procedure, for
diagonal and nonnegative quadratic functionals. We consider smoothness classes of ellipsoidal form and check that our conditions are
fulfilled in the particular case of ellipsoids corresponding to anisotropic Sobolev classes. In this case, we present a closed form of
the minimax rate and the separation constant. On the other hand, minimax rates for quadratic functionals which are neither positive nor
negative makes appear two different regimes: ``regular'' and ``irregular''. In the ``regular" case, the minimax rate is equal to $n^{-1/4}$
while in the ``irregular'' case, the rate depends on the smoothness class and is slower than in the ``regular'' case. We apply this to the
problem of testing the equality of Sobolev norms of two functions observed in noisy environments.
\end{abstract}


\end{frontmatter}

\section{Introduction}

\subsection{Problem statement}
Consider the nonparametric regression model with multi-dimensional random design: We observe $(x_i,\bt_i)_{i=1,\ldots,n}$ obeying the relation
\begin{equation}\label{eq:1}
x_i=f(\bt_i)+\xi_i, \quad i=1,\ldots, n,
\end{equation}
where $\bt_i\in \Delta\subset \RR^d$ are random design points, $1\leq d<\infty$, $f:\Delta\to\RR$ is the unknown regression function and $\xi_i$s
represent observation noise. Throughout this work, we assume that the vectors $\bt_i=(t_i^1,\ldots,t_i^d)$, for $i=1,\ldots,n$, are independent and identically
distributed with uniform distribution on $\Delta=[0,1]^d$, which is equivalent to $t_i^k\stackrel{\text{iid}}{\sim}\mathcal{U}(0,1)$. Furthermore,
conditionally on $\mcT_n=\{\bt_1, \ldots, \bt_n\}$, the variables $\xi_1,\ldots,\xi_n$ are assumed i.i.d.\ with zero mean and variance $\tau^2$,
for some known $\tau\in(0,\infty)$.

Let $L_2(\Delta)$ denote the Hilbert space  of all squared integrable functions defined on $\Delta$. Assume that we are given two
disjoint subsets $\mcF_0$ and $\mcF_1$ of $L_2(\Delta)$. We are interested in analyzing the problem of testing hypotheses:
\begin{equation}\label{eq:2}
H_0 : f\in \mcF_0\qquad\text{against}\qquad H_1 : f\in \mcF_1.
\end{equation}
To be more precise, let us set $\bz_i=(x_i,\bt_i)$ and denote by $P_f$ be the probability distribution of the data vector $(\bz_1,\ldots,\bz_n)$
given by (\ref{eq:1}). The expectation with respect to $P_f$ is denoted by $E_f$. The goal is to design a testing procedure
$\phi_n:(\RR\times\Delta)^n\to \{0,1\}$ for which we are able to establish theoretical guarantees in terms of the cumulative error rate
(the sum of the probabilities of type I and type II errors):
\begin{equation}\label{eq:3}
\gamma_n(\mcF_0, \mcF_1, \phi_n)=\underset{f\in\mcF_0}{\sup}P_f(\phi_n=1)+\underset{f\in\mcF_1}{\sup}P_f(\phi_n=0).
\end{equation}
To measure the statistical complexity of this testing problem, it is relevant to analyze the minimax error rate
\begin{equation}\label{eq:4}
\gamma_n(\mcF_0,\mcF_1) = \inf_{\phi_n} \gamma_n(\mcF_0, \mcF_1, \phi_n),
\end{equation}
where $\inf_{\phi_n}$ denotes the infimum over all testing procedures.

The focus in this paper is on a particular type of null hypotheses $H_0$ that can be defined as the set of functions lying in the
kernel of some quadratic functional $Q:L_2(\Delta)\to\RR$, \textit{i.e.},
$\mcF_0\subset \big\{f \in  L_2(\Delta) : Q[f]=0\big\}$.
As described later in this section, this kind of null hypotheses naturally arises in several problems including variable selection,
testing partial linearity of a regression function or the equality of norms of two signals. Then, it is appealing to define the
alternative as the set of functions satisfying $|Q[f]|>\rho^2$ for some $\rho>0$. However, without further assumptions on the nature
of functions $f$, it is impossible to design consistent testing procedures for discriminating between $\mcF_0$ and $\mcF_1$. One
approach to making the problem meaningful is to assume that the function $f$ belongs to a smoothness class. Typical examples of
smoothness classes are Sobolev and H\"older classes, Besov bodies or balls in reproducing kernel Hilbert spaces.

In the present work, we assume that the function $f$ belongs to a smoothness class $\Sigma$ that can be seen as an ellipsoid in
the infinite-dimensional space $L_2(\Delta)$. Thus, the null and the alternative are defined by
\begin{equation}\label{eq:5}
\mcF_0 = \Big\{f \in  \Sigma : Q[f]=0\Big\},\qquad \mcF_1 = \mcF_1(\rho)=\Big\{f \in  \Sigma : |Q[f]|\ge \rho^2\Big\}.
\end{equation}
One can take note that both hypotheses are composite and nonparametric.

\subsection{Background on minimax rate- and sharp-optimality}

Given the observations $(x_i,\bt_i)_{i=1,\ldots,n}$, we consider the problem of testing the composite hypothesis $\mcF_0$
against the nonparametric alternative $\mcF_1(\rho)$ defined by (\ref{eq:5}). The goal here is to obtain, if possible, both rate and sharp
asymptotics for the cumulative error rate in the minimax setup. These notions are defined as follows. For a fixed small number
$\gamma\in(0,1)$, the function $r_n^*$ is called minimax rate of testing if:
\begin{itemize}
\item there exists $C'>0$  such that  $\forall\,C< C'$, we have
$\liminf_{n\to \infty}\limits  \gamma_n(\mcF_0, \mcF_1(Cr_n^*))\geq \gamma$,
\item there exists $C^{''}>0$ and a test $\phi_n$ such that $\forall\,C> C^{''}$,
$\limsup_{n\to \infty}\limits  \gamma_n(\mcF_0, \mcF_1(Cr_n^*), \phi_n)\leq \gamma$.
\end{itemize}
A testing procedure $\phi_n$ is called minimax rate-optimal if $\limsup_{n\to \infty}\gamma_n(\mcF_0, \mcF_1(Cr_n^*), \phi_n)\leq \gamma$ for some
$C>0$. Note that the minimax rate and the rate-optimal test may depend on the prescribed significance
level $\gamma$. However, in most situations this dependence cancels out from the rate and appears
only in the constants. If the constants $C'$ and $C''$ coincide, then their common value is called
exact separation constant and any test satisfying the second condition is called minimax sharp optimal.
The minimax rate $r_n^*$ is actually not uniquely defined, but the product of the minimax rate with the
exact separation constant is uniquely defined up to an asymptotic equivalence.
For more details on minimax hypotheses testing we refer to \citep{ingster2003nonparametric}.

While minimax rate-optimality is a desirable feature for a testing procedure, it may still lead to
overly conservative tests. A (partial) remedy for this issue is to consider sharp asymptotics of the
error rate. In fact, one can often prove that when $n\to\infty$,
\begin{equation}\label{gauss}
\gamma_n(\mcF_0, \mcF_1(\rho))=2\Phi(-u_n(\rho))+o(1),
\end{equation}
where $\Phi$ is the c.d.f.\ of the standard Gaussian distribution, $u_n(\cdot)$ is some ``simple'' function from
$\RR_+$ to $\RR$ and $o(1)$ is a term tending to zero uniformly in $\rho$ as $n\to\infty$. This relation
implies that by determining $r_n^*$ as a solution with respect to $\rho$ to the equation
$u_n(\rho)=z_{1-\gamma/2}$---where $z_\alpha$ stands for the $\alpha$-quantile of the standard Gaussian
distribution---we get not only the minimax rate, but also the exact separation constant. When relation
(\ref{gauss}) is satisfied, we say that Gaussian asymptotics hold.

\subsection{Overview of the main contributions}

Our contributions focus on the case where the smoothness class $\Sigma$ is an ellipsoid in $L_2(\Delta)$ and
the quadratic functional $Q$ admits a diagonal form in the orthonormal basis corresponding to the directions
of the axes of the ellipsoid $\Sigma$. To be more precise, let $\mathcal{L}$ be a countable set and
$\{\varphi_l\}_{l\in\mathcal{L}}$ be an orthonormal system in $L_2(\Delta)$. For a function $f\in L_2(\Delta)$,
let $\ttheta[f]=\{\theta_l[f]\}_{l\in\mathcal{L}}$ be the generalized Fourier coefficients with respect to
this system, \textit{i.e.}, $\theta_l[f]=\langle f,\varphi_l\rangle$, where $\langle\cdot,\cdot\rangle$ denotes the inner
product in $L_2(\Delta)$. The functional sets  $\Sigma\subset L_2(\Delta)$ under consideration are subsets of ellipsoids
with directions of axes $\{\varphi_l\}_{l\in\mathcal{L}}$ and with coefficients $\bc=\{c_l\}_{l\in\mathcal{L}}\in\RR_+^{\mathcal{L}}$:
\begin{equation}\label{eq:6}
\Sigma\subset\Big\{f=\sum_{l\in\mathcal{L}}\nolimits \theta_l[f]\varphi_l:\ \sum_{l\in\mathcal{L}}\nolimits c_l\theta_l[f]^2\leq 1\Big\}.
\end{equation}
The diagonal quadratic functional is defined by a set of coefficients $\bq=\{q_l\}_{l\in\mathcal{L}}$:
$Q[f]=\sum_{l\in\mathcal{L}}\nolimits q_l\theta_l[f]^2$.
Note that if $Q$ is definite positive, \textit{i.e.}, $q_l>0$ for all $l\in\mcL$, then the null hypothesis
becomes $f=0$ and the problem under consideration is known as detection problem. However, the goal of the present
work is to consider more general types of diagonal quadratic functionals. Namely, two situations are examined: (a)
all the coefficients $q_l$ are nonnegative and (b) the two sets $\mcL_+=\{l\in\mcL:q_l>0\}$ and $\mcL_-=\{l\in\mcL:q_l<0\}$
are nonempty.

In the first situation, we establish Gaussian asymptotics of the cumulative error rate and propose a minimax sharp-optimal
test. Under some conditions, we show that the sequence\footnote{We denote by ${\|\cdot\|}_2$ and by $\langle\cdot,\cdot\rangle$
the usual norm and the inner product in $\ell_2(\mcL)$, the space of squared summable arrays indexed by $\mcL$.}
\begin{equation}\label{eq:8}
r_{n,\gamma}^*=\min\Big\{\rho>0: \inf_{\bv\in\RR_+^\mcL:\langle \bv,\bc\rangle \le 1;\langle \bv,\bq\rangle \ge \rho^2} \|\bv\|_2^2\ge 8n^{-2}z_{1-\gamma/2}\Big\}
\end{equation}
provides the minimax rate of testing with constants $C'=C''=1$. This result is instantiated to some examples motivating our interest for
testing the hypotheses (\ref{eq:5}). One example, closely related to the problem of variable selection~\citep{CD11a}, is testing the
relevance of a particular covariate in high-dimensional regression. This problem is considered in a more general setup corresponding to
testing that a partial derivative of order $\aalpha=(\alpha_1,\ldots,\alpha_d)$, denoted by
$\partial^{\alpha_1+\ldots+\alpha_d}f/\partial t_1^{\alpha_1}\ldots\partial t_d^{\alpha_d}$, is identically equal to zero against the hypothesis
that this derivative is significantly different from 0.
As a consequence of our main result, we show that if $f$ lies in the anisotropic Sobolev ball of smoothness $\ssigma=(\sigma_1,\ldots,\sigma_d)$,
and we set $\delta=\sum_{i=1}^d\alpha_i/\sigma_i$, $\bar\sigma=\big(\frac1d\sum_{i=1}^d\sigma_i^{-1}\big)^{-1}$, then the minimax optimal-rate is
$r_n^*=n^{-2\bar\sigma(1-\delta)/(4\bar\sigma+d)}$ provided that $\delta<1$ and $\bar\sigma>d/4$. Furthermore, we derive Gaussian asymptotics and
exhibit the exact separation constant in this problem.

The second situation we examine in this paper concerns the case where the cardinalities of both $\mcL_+$ and $\mcL_-$ are nonzero. A typical
application of this kind of problem is testing the equality of the norms of two signals observed  in noisy environments. In this set-up, we
provide minimax rates of testing and exhibit the presence of two regimes that we call regular regime and irregular regime. In the regular regime,
the minimax rate is  $r_n^*=n^{-1/4}$, while in the irregular case it may be of the form $n^{-a}$ with an $a<1/4$ that depends on the degree of
smoothness of the functional class.

Note that all our results are non-adaptive: our testing procedures make explicit use of the smoothness characteristics of the function $f$.
Adaptation to the unknown smoothness for the problem we consider is an open question for which the works \citep{Spok96,Gayraud_Pouet2} may
be of valuable guidance.

\subsection{Relation to previous work}

Starting from the seminal papers by \citet{Ermakov90} and \citet{Ingster93_1,Ingster93_2,Ingster93_3}, minimax
testing of nonparametric hypotheses received a great deal of attention. A detailed review of the literature
on this topic being out of scope of this section, we only focus on discussing those previous results which
are closely related to the present work.  The goal here is to highlight the common points and the most striking
differences with the existing literature. The major part of the statistical inference for nonparametric hypotheses testing was developed for
the Gaussian white noise model (GWNM) and its equivalent formulation as Gaussian sequence model (GSM). As recent
references for the problem of testing a simple hypothesis in these models, we cite \citep{Ermakov11,Ingster_12},
where the reader may find further pointers to previous work. In the present work, the null hypothesis defined by
(\ref{eq:5}) is composite and nonparametric. Early references for minimax results for composite null hypotheses
include  \citep{Horowitz01,Pouet01,Gayraud_Pouet1,Gayraud_Pouet2}, where the case of parametric null hypothesis is
of main interest. These papers deal with the one-dimensional situation and provide only minimax rates of testing
without attaining the exact separation constant. Furthermore, the alternative is defined as the set of functions
that are at least at a Euclidean distance $\rho$ from the null hypothesis, which is very different from the alternatives
considered in this work.

More recently, the nonasymptotic approach to the minimax testing gained popularity \citep{baraud2003adaptive,Baraudetal05,
laurent2011testing,laurent2012non}. One of the advantages of the nonasymptotic approach is that it removes
the frontier between the concepts of parametric and nonparametric  hypotheses, while its limitation is that
there is no result on sharp optimality (even the notion itself is not well defined). Note also that
all these papers deal with the GSM considering as main application the case of one dimensional signals,
as opposed to our set-up of regression with high-dimensional covariates.

Let us review in more details the papers \citep{ingster2009minimax} and \citep{laurent2011testing} that
are very closely related to our work either by the methodology which is used or by the problem of interest.
\cite{ingster2009minimax} extended some results on the goodness-of-fit testing for the $d$-dimensional GWNM to
the goodness-of-fit testing  for the multivariate nonparametric regression model. More precisely, they tested
the null hypothesis $H_0 : f=f_0$, where $f_0$ is a known function, against the alternative
$H_1 : f\in\Sigma, \int_\Delta (f-f_0)^2\geq r_n^2$, where $\Sigma$ is an ellipsoid in the Hilbert space $L_2(\Delta)$.
They obtained both rate and sharp
asymptotics for the error probabilities in the minimax setup. So the model they considered is the same as the one we
are interested in here, but the hypotheses $H_0$ and $H_1$ are substantially different. As a consequence, the testing
procedure we propose takes into account the general forms of $H_0$ and $H_1$ given by (\ref{eq:5}) and is different from
the asymptotically minimax test of \cite{ingster2009minimax}. Furthermore, we substantially relaxed the contraint on
the noise distribution by replacing Gaussianity assumption by the condition of bounded 4th moment.

\cite{laurent2011testing} considered the GWNM from the inverse problem point of view, \textit{i.e.}, when the signal of
interest $g$ undergoes a linear transformation $T$ before being observed in noisy
environment. This corresponds to  $f=T[g]$ with a compact injective operator $T$.
Then the two assertions  $g=0$ and  $T[g]=0$ are equivalent. Consequently, if the goal is to detect the signal $f$, one
can consider the two testing problems :
\begin{enumerate}
\item (inverse formulation) $H_0 : T^{-1}[f]=0$ against $H_1 : \|T^{-1}[f]\|_2\geq \rho$.
\item (direct formulation) $H_0 : f=0$ against $H_1 : \|f\|_2\geq \rho$.
\end{enumerate}
The authors discussed advantages and limitations of each of these two formulations in terms of minimax rates.
Depending on the complexity of the inverse problem and on the assumptions on the function to be detected
(sparsity or smoothness), they proved that the specific treatment devoted to inverse problem which includes an
underlying inversion of the operator, may worsen the detection accuracy. For each situation, they also highlighted
the cases where the direct strategy fails while a specific test for inverse formulation works well. The inverse formulation
is closely related to our definition (\ref{eq:5}) of the hypotheses $H_0$ and $H_1$, since
$Q[f]=\|T^{-1}[f]\|_2^2$ is a quadratic functional. However, our setting is more general in that we consider functionals
with non-trivial kernels and with possibly negative diagonal entries.

\subsection{Organization}

The rest of the paper is organized as follows. The results concerning sharp asymptotics for
positive semi-definite diagonal functionals are provided in Section \ref{sec:2}. In particular,
the rates of separation for a general class of tests called linear U-tests are explored in Subsection~\ref{ssec:2.2}.
The asymptotically optimal linear U-test is provided in Subsection~\ref{ssec:2.3} along with its rate of separation, which
is shown to coincide with the minimax exact rate in Subsection~\ref{ssec:2.4}. Section~\ref{sec:examples} is devoted to
a discussion of the assumptions and to the consequences of the main result for some relevant examples.
The results for nonpositive and nonnegative diagonal quadratic functionals are stated in Section~\ref{sec:NPF} along with
an application to testing the equality of the norms of two signals. A summary and some perspectives are provided in Section~\ref{sec:conc}.
Finally, the proofs of the results are postponed to the Appendix.

\section{Minimax testing for nonnegative quadratic functionals}\label{sec:2}

\subsection{Additional notation}
In what follows, the notation $A_n=O(B_n)$ means that there exists a constant $c>0$ such that $A_n\leq c B_n$ and the notation
$A_n=o(B_n)$ means that the ratio $A_n/B_n$ tends to zero.  The relation $A_n\sim B_n$ means that $A_n/B_n$ tends to 1, while
the relation $A_n\asymp B_n$ means that there exist constants $0<c_1<c_2<\infty$ and $n_0$ large enough such that
$c_1\leq A_n/B_n\leq c_2$ for $n\geq n_0$. For a real number $c$,  we denote by $c_+$ its positive part
$\max (0,c)$ and by $\lfloor c \rfloor$ its integer part. For a set $\mcA$, $\1_A$ stands for its indicator function and
$|\mcA|$ denotes its cardinality. Given a $q>0$ and a function $f$, $\|f\|_q=\big(\int_{\Delta} |f(t)|^qdt\big)^{1/q}$ is
the conventional $\ell_q$-norm of $f$.  Similarly, for a vector or an array $\bu$ indexed by a countable set $\mcL$,
$\|\bu\|_q=(\sum_{l\in\mcL} |u_l|^q)^{1/q}$ is the $\ell_q$-norm of $\bu$. As usual,
we also denote by $\|\bu\|_0$ and $\|\bu\|_\infty$, respectively, the number of nonzero entries and the magnitude of the largest
entry of $\bu\in\RR^{\mcL}$.

In the sequel, without loss of generality, we assume that the standard deviation of the noise is equal to one: $\tau=1$.
The case of general but known $\tau$ can be deduced as a consequence of our results.

Recall that we consider quadratic functionals $Q$ of the form $Q[f]=\sum_{l\in\mathcal{L}} q_l\theta_l[f]^2$, for some given
array $\bq=\{q_l\}_{l\in\mathcal{L}}$. The major difference between the functional $\sum_{l\in\mcL} \theta_l[f]^2$ that appears
in the problem of detection \citep{ingster2009minimax,Ingster_12} and this general functional actually lies in the fact that
the support of $\bq$ defined by $S_F=\text{supp}(\bq)=\big\{l\in\mcL : q_l\neq 0\big\}$ is generally different from
$\mathcal{L}$. Furthermore, large coefficients $q_l$ amplify the error of estimating $Q[f]$ and, therefore, it becomes more difficult
to distinguish $H_0$ from $H_1$. An interesting question, to which we answer in the next sections, is what is the interplay
between $\bc$ and $\bq$ that makes it possible to distinguish between the null and the alternative.

Let  $S_F^c$ denote the complement of $S_F$ and, for a set $L\subset \mathcal{L}$,  span$\big(\{\varphi_l\}_{l\in L}\big)$ be
the closed linear subspace of $L_2(\Delta)$ spanned by the set $\{\varphi_l\}_{l\in L}$.  Let $\Pi_{S_F}f$ and $\Pi_{{S_F^c}}f$
be the orthogonal projections of a function $f\in \Sigma$ on $\text{span}\big(\{\varphi_l\}_{l\in {S_F}}\big)$ and
$\text{span}\big(\{\varphi_l\}_{l\in {S_F^c}}\big)$ respectively. To simplify notation, the subscript ${S_F^c}$ is
omitted in the rest of the paper, \textit{i.e.}, $\Pi_{{S_F^c}}f$ is replaced by $\pf$.
Finally, throughout this work we will assume that  $f$ is centered, \textit{i.e.}, $\int_\Delta f(\bt)\,d\bt=0$, and
that $\{\varphi_l\}$ is an orthonormal basis of the subspace of $L_2(\Delta)$ consisting of all centered functions. In other terms,
all the functions $\varphi_i$ are orthogonal to the constant function.

\subsection{Linear U-tests and their error rate}\label{ssec:2.2}

We start by introducing a family of testing procedures that we call linear U-tests. To this end, we split the sample into two parts:
a small part of the sample is used to build a pilot estimator $\proj$ of $\pf$, whereas the remaining observations are used for
distinguishing between $H_0$ and $H_1$. Let us set $m=n-\lfloor \sqrt{n}\rfloor$ and call the two parts of the sample
$\mcD_1=\{(x_i,\bt_i):i=1,\ldots,m\}$ and $\mcD_2=\{(x_i,\bt_i):i=m+1,\ldots,n\}$.
Using a pilot estimator $\proj$ of $\pf$, we define the adjusted observations $\tilde{x}_i=x_i-\proj(\bt_i)$ and $\tilde{\bz}_i=(\tilde{x}_i, \bt_i)$.

\begin{definition}
Let $\bw_n=\{w_{l,n}\}_{l\in S_F}$ be an array of real numbers containing a finite number of nonzero entries and such that $\|\bw_n\|_2=1$.
Let $u$ be a real number. We call a linear U-test based on the array $\bw_n$
the procedure $\phi^\bw_n=\1_{\{U^\bw_n>u\}}$, where $U_n$ is the linear in $\bw_n$ U-statistic defined by
\begin{align}\label{U-stat}
U_n=\bigg(\frac{2}{m(m-1)}\bigg)^{1/2}\sum_{1\leq i<j \leq m} \tilde{x}_i\tilde{x}_j \sum_{l\in S_F} w_{l,n}\varphi_l(\bt_i)\varphi_l(\bt_j).
\end{align}
\end{definition}

We shall prove that an appropriate choice of $\bw_n$ and $u$ leads to a linear U-test that is asymptotically sharp-optimal.
The rationale behind this property relies on the by now well-understood principle of smoothing out high frequencies of a noisy
signal. In fact, if we call $\{\theta_l[f]\}_{l\in S_F}$ the (relevant part of the) representation of $f$ in the frequency
domain, then $\{\frac1m\sum_{i=1}^m \tilde x_i\varphi_l(\bt_i)\}_{l\in S_F}$ is a nearly unbiased estimator of this representation.
Then, the array $\bw_n$ acts as a low pass filter that shrinks to zero the coefficients corresponding to high frequencies in order
to prevent over-fitting.

The first step in establishing theoretical guarantees on the error rate of a linear U-test consists in exploring the behavior of
the statistic $U_n$ under the null.

\begin{proposition}\label{prop_1}
Let $w_{n,l}\ge 0$ for all $n\in\NN$ and $l\in\mcL$.
Assume that $E[\xi_1^4]<\infty$ and the following conditions are fulfilled:
\begin{itemize}
\item For some $C_w<\infty$, $\|\bw_n\|^2_\infty \|\bw_n\|_0\le C_w$.
\item As $n\to\infty$, $\|\bw_n\|_0 \to\infty$ so that $ \|\bw_n\|_0 = o(n)$.
\item For some $C_\varphi<\infty$,
$\sup_{\bt\in\Delta}\sum_{l:w_{l,n}\not=0} \varphi_l^2(\bt)\le C_\varphi \|\bw_n\|_0$.
\item As $n\to\infty$, $\sup_{f\in \Sigma} E_f[\|\pf-\proj\|_4^4]=o(1)$.
\end{itemize}
Then, uniformly in $f\in\mcF_0$, the U-statistic defined by (\ref{U-stat}) converges in distribution to the standard
Gaussian distribution $\mcN(0,1)$.
\end{proposition}

In other terms, this proposition claims that under appropriate conditions, for every $u\in\RR$, the sequence
$\sup_{f\in\mcF_0}|P_f(U_n>u)-\Phi(u)|$ tends to zero, as $n$ goes to infinity. This means that under the null, the distribution of the
test statistic $U_n$ is asymptotically parameter free. This is frequently referred to as Wilks' phenomenon.

To complete the investigation of the error rate of a linear U-test, we need to characterize the behavior of the
test statistic $U_n$ under the alternative. As usual, this step is more involved. Roughly speaking, we will show
that under the alternative the test statistic $U_n$ is close to a Gaussian random variable with mean
$h_n[f,\bw_n] = \big(\frac{m(m-1)}{2}\big)^{1/2}\sum_{l\in \mcL(\bw_n)}\nolimits w_{l,n} \theta_l^2[f]$ and variance 1.
The rigorous statement is provided in the next proposition.

\begin{proposition}\label{prop_2}
Let the assumptions of Proposition~\ref{prop_1} be satisfied. Assume that in addition:
\begin{itemize}
\item There exists a sequence $\zeta_n$ such that $\zeta_n^{-1}=o(n)$ and $\sup_{l\in S_F:w_{l,n}<\zeta_n} c_l^{-1}=o(1)$.
\item For some $p>4$, we have $\sup_{f\in\Sigma} \|\Pi_{S_F}f\|_p<\infty$.
\end{itemize}
Then, for every $\rho>0$, the type II error of the linear U-test based on $\bw_n$ satisfies:
\begin{align}
\sup_{f\in\mcF_1(\rho)}\nolimits P_f(\phi_n^\bw=0)\le \sup_{f\in\mcF_1(\rho)}\nolimits \Phi(u-h_n[f,\bw_n]) + o(1),
\end{align}
where the term $o(1)$ does not depend on $\rho$.
\end{proposition}

Let us provide an informal discussion of the assumptions introduced in the previous propositions. The first two assumptions
in Proposition~\ref{prop_1} mean that most nonzero entries of the array $\bw_n$  should be of the same order. Arrays that
have a few spikes and many small entries are discarded by these assumptions. Furthermore, the number of samples in the frequency domain
that are not annihilated by $\bw_n$ should be small as compared to the sample size $n$. The third assumption of Proposition~\ref{prop_1}
is trivially satisfied for bases of bounded functions such as sine and cosine bases and their tensor products. For localized bases
like wavelets, this assumption imposes a constraint on the size of the support of $\bw_n$: it should not be too small. The last assumption
of Proposition~\ref{prop_1} will be discussed in more detail later. One should also take note that the only reason for requiring
from the functions $f$ to be smooth under the null is the need to be able to construct a uniformly consistent pilot estimator of $\pf$.

Concerning the assumptions imposed in Proposition~\ref{prop_2}, the first one means that only coefficients $\theta_l$ corresponding
to high frequencies are strongly shrunk by $\bw_n$. This is a kind of coherence assumption between the smoothing filter $\bw_n$ and the
coefficients $\bc=\{c_l\}_{l\in\mcL}$ encoding the prior information on the signal smoothness. The second assumption of Proposition~\ref{prop_2}
is rather weak and usual in the context of regression with random design. It is only needed for getting uniform control of the error rate and
the actual value of the norm $\|\Pi_{S_F}f\|_p$ does not enter in any manner in the definition of the testing procedure.

Let us draw now the consequences of the previous propositions on the cumulated error rate of a linear U-test. Using the monotonicity
of the Gaussian c.d.f.\ $\Phi$, under the assumptions of Proposition~\ref{prop_2}, we get
\begin{align}\label{eq:19}
\gamma_n(\mcF_0,\mcF_1(\rho),\phi_n^\bw)\le \Phi(-u)+ \Phi\Big(u-\inf_{f\in\mcF_1(\rho)}\nolimits h_n[f,\bw_n]\Big)+o(1),
\end{align}
where the term $o(1)$ is uniform in $\rho>0$. Using the symmetry of $\Phi$ and the monotonicity of $\Phi'$ on $\RR_+$, one easily checks that
the value of the threshold $u$ minimizing the main term in the right-hand side of the last display is
$u=\frac12\inf_{f\in\mcF_1(\rho)}h_n[f,\bw_n].$
This result provides a constructive tool for determining the rate of separation
of a given linear U-test. In fact, one only needs to set $u=z_{1-\gamma/2}$ and find a sequence $r_n$ such that
$\inf_{f\in\mcF_1(r_n)}h_n[f,\bw_n]\sim 2z_{1-\gamma/2}$, where $z_\alpha$ is the $\alpha$-quantile of $\mcN(0,1)$.

\begin{remark}\label{tildexi}
We explain here the use of $\tilde{x}_i$ instead of $x_i$ in our testing procedure. Actually if we were only interested in rate-optimality,
this precaution would not have been necessary. The problem only arises when dealing with sharp-optimality and it concerns the variance of $U_n$.
Indeed we need some terms that appear in the variance to tend to zero when $Q[f]=0$ or $Q[f]$ is small (those terms only need to be bounded for
the rate-optimality). If we had used $x_i$ instead of $\tilde{x}_i$, we would have ended up with terms like $\|f\|_2$ in the variance. The
information contained in the assertion ``$Q[f]$ is small'' concerns only the coefficients $\{\theta_l\}_{l\in {S_F}}$, thus it implies that
$\|\Pi_{S_F}f\|_2$ is small but it does not say anything about $\|f\|_2$. We can also remark that this problem does not arise
in the Gaussian sequence model as one estimates $\theta_l^2$ by an unbiased estimator whose variance  makes appear only $\theta_l$.
\end{remark}

\begin{remark}\label{criterion}
We chose to consider only the criterion $\gamma_n(\mcF_0,\mcF_1(\rho), \phi_n^\bw)$ so as to simplify the exposition of our results.
But we could have dealt with the classical Neyman-Pearson criterion that we recall here. For a significance level $0<\alpha<1$ and a test $\psi$,
we set
$$
\alpha(\mcF_0,{\psi})=\sup_{f\in\mcF_0}\nolimits P_f(\psi=1),\quad
\beta(\mcF_1, \psi)=\inf_{\psi}\nolimits\sup_{f\in\mcF_1}\nolimits P_f(\psi=0),
$$
Instead of the minimax risk $\gamma_n(\mcF_0,\mcF_1(\rho))$ we could have considered the quantity
$\beta_n(\mcF_0, \mcF_1(\rho))=\inf_{\psi : \alpha(\mcF_0,\psi)\leq \alpha}\beta(\mcF_1(\rho), \psi)$. This criterion is
considered in \cite{ingster2009minimax} and more generally in \cite{ingster2003nonparametric}. The transposition to  our
case is straightforward.
\end{remark}

\subsection{Minimax linear U-tests}\label{ssec:2.3}

The  relation  (\ref{eq:19}) being valid for a large variety of arrays $\bw_n$, it is natural to
look for a $\bw_n$ minimizing the right-hand side of (\ref{eq:19}). This leads to the following saddle point problem:
\begin{align}\label{saddle}
\sup_{\substack{\bw\in\RR^{\mcL}_+\\ \|\bw\|_2=1}}\inf_{f\in\mcF_1(\rho)} \sum_{l\in\mcL} w_l \theta_l[f]^2=
\sup_{\substack{\bw\in\RR^{\mcL}_+\\ \|\bw\|_2=1}}
\inf_{\substack{\bv\in\RR^{\mcL}_+\\ \langle \bv,\bc\rangle \le 1, \langle \bv,\bq\rangle \ge \rho^2 }}
\langle \bw ,\bv\rangle .
\end{align}
It turns out that this saddle point problem can be solved with respect to $\bw$ and leads to a one-parameter family of
smoothing filters $\bw$.

\begin{proposition}\label{prop_3}
Assume that for every $T>0$, the set $\NL(T)=\{l\in S_F: c_l<Tq_l\}$ is finite. For a given $\rho>0$, assume that the equation
\begin{equation}
\frac{\sum_{l\in\mcL} q_l(Tq_l-c_l)_+}{\sum_{l\in\mcL} c_l(Tq_l-c_l)_+}=\rho^2
\end{equation}
has a solution and denote it by $T_\rho$. Then, the pair $(\bw^*,\bv^*)$ defined by
\begin{align}\label{eq:22}
v_l^*=\frac{(T_\rho q_l-c_l)_+}{\sum_{l\in\mcL}c_l(T_\rho q_l-c_l)_+}\qquad w_l^*=\frac{v_l^*}{\|\bv^*\|_2}
\end{align}
provides a solution to the saddle point problem (\ref{saddle}), that is
\begin{align*}
\langle \bw^*,\bv^*\rangle=
\sup_{\substack{\bw\in\RR^{\mcL}_+\\ \|\bw\|_2=1}}\nolimits
\inf_{\substack{\bv\in\RR^{\mcL}_+\\ \langle \bv,\bc\rangle \le 1, \langle \bv,\bq\rangle \ge \rho^2 }}
\langle \bw ,\bv\rangle= \inf_{\substack{\bv\in\RR^{\mcL}_+\\ \langle \bv,\bc\rangle \le 1, \langle \bv,\bq\rangle \ge \rho^2 }}
\langle \bw^* ,\bv\rangle.
\end{align*}
\end{proposition}

This result tells us that the ``optimal'' weights $\bw_n$ for the linear U-test $\phi_n^\bw$ should be of the
form (\ref{eq:22}), which is particularly
interesting because of its dependence on only one parameter $T>0$. The next theorem provides a simple strategy for
determining the minimax sharp-optimal test among linear U-tests satisfying some mild assumptions. We will show later
in this section that this test is also minimax sharp-optimal among all possible tests.

\begin{theorem}\label{thm_1}
Assume that $E[\xi_1^4]<\infty$ and for every $T>0$, the set $\NL(T)=\{l\in S_F: c_l<Tq_l\}$ is finite. For a prescribed
significance level $\gamma\in(0,1)$, let $T_{n,\gamma}$ be a sequence of positive numbers such that the following relation
holds true: as $n\to\infty$,
\begin{equation} \label{eq:24}
\bigg(\frac{m(m-1)}{2}\sum_{l\in\mcL} (T_{n,\gamma}q_l-c_l)_+^2\bigg)^{1/2} = \bigg(\sum_{l\in\mcL} c_l(T_{n,\gamma}q_l-c_l)_+\bigg)(2z_{1-\gamma/2}+o(1)).
\end{equation}
Let us define
\begin{equation}\label{eq:25}
r_{n,\gamma}^*=\bigg\{\frac{\sum_{l\in\mcL} q_l(T_{n,\gamma}q_l-c_l)_+}{\sum_{l\in\mcL} c_l(T_{n,\gamma}q_l-c_l)_+}\bigg\}^{1/2}.
\end{equation}
If the following conditions are fulfilled:
\begin{itemize}
\item[]{\bf {[C1]}} For some constant $C_1>0$,
$ |\mcN(T_{n,\gamma})|\max_{l\in\mcN(T_{n,\gamma})}q_l^2\le C_1 \sum_{l\in\mcN(T_{n,\gamma})} \big(q_l-\frac{c_l}{T_{n,\gamma}}\big)^2$.\\[-8pt]
\item[]{\bf {[C2]}} As $n\to\infty$, $\sum_{l\in\mcN(T_{n,\gamma})} q_l^2= o(n^2\min_{l\in\mcN(T_{n,\gamma})}q_l^2)$.\\[-8pt]
\item[]{\bf {[C3]}} For some constant $C_3>0$, $\sup_{\bt\in\Delta}\sum_{l\in\mcN(T_{n,\gamma})}\varphi_l^2(\bt)\le C_3 |\mcN(T_{n,\gamma})|$.\\[-8pt]
\item[]{\bf {[C4]}} As $n\to\infty$, $|\mcN(T_{n,\gamma})| \to\infty$ so that $|\mcN(T_{n,\gamma})|=o(n)$.\\[-8pt]
\item[]{\bf {[C5]}} As $n\to\infty$, $T_{n,\gamma}\inf_{l\in S_F} q_l$ tends to $+\infty$.\\[-8pt]
\item[]{\bf {[C6]}} As $n\to\infty$, $\sup_{f\in\Sigma} E_f[\|\pf-\proj\|_4^4]=o(1)$.\\[-8pt]
\item[]{\bf {[C7]}} For some $p>4$, it holds that $\sup_{f\in\Sigma}\|\Pi_{S_F}f\|_p<\infty$.
\end{itemize}
then the linear U-test $\widehat\phi_n^*=\1_{\{U_n^{\widehat\bw^*}>z_{1-\gamma/2}\}}$ based on the array $\widehat\bw_n^*$ defined by
$$
\widehat w_{l,n}^*=\frac{(T_{n,\gamma}q_l-c_l)_+}{\big[\sum_{l'\in\mcL}(T_{n,\gamma}q_{l'}-c_{l'})_+^2\big]^{1/2}}
$$
satisfies
\begin{equation}\label{eq:26}
\gamma_n(\mcF_0,\mcF_1(r_{n,\gamma}^*),{\widehat\phi}_n^*) \le \gamma+o(1),\qquad\text{as}\qquad n\to\infty.
\end{equation}
\end{theorem}

The proof of this result, provided in the Appendix, is a direct consequence of Proposition~\ref{prop_1}, \ref{prop_2} and \ref{prop_3}.
As we shall see below, the rate $r_{n,\gamma}^*$ defined in Theorem~\ref{thm_1} is the minimax sharp-rate in the problem of testing
hypotheses (\ref{eq:5}), provided that the assumptions of the theorem are fulfilled. As expected, getting such a strong
result requires non-trivial assumptions on the nature of the functional class, that of the hypotheses to be tested, as well as the
interplay between them. Some short comments on these assumptions are provided in the remark below, with a further development
left to subsequent sections.

\begin{remark}
The very first assumption is that the set $\mcN(T)$ is finite. It is necessary for ensuring that the linear U-test we introduced is computable.
This assumption is fulfilled when, roughly speaking, the coefficients which express the regularity, $\{c_l\}_{l\in\mathcal{L}}$, grow at a
faster rate than the coefficients $\{q_l\}_{l\in\mathcal{L}}$ of the quadratic functional $Q$.
Assumptions {\bf [C1]}, {\bf [C2]}, {\bf [C4]} and {\bf [C5]} are satisfied in most cases we are interested in. Two illustrative
examples---concerning Sobolev ellipsoids with quadratic functionals related to partial derivatives---for which these hypotheses are
satisfied are presented in Subsections~\ref{ssec:3.2} and \ref{ssec:3.3}. Assumption {\bf [C3]} is essentially a constraint on the
basis $\{\varphi_l\}$; we show in Subsection \ref{ssec:3.1} that it is satisfied by many bases commonly used in statistical literature.
{\bf [C6]} and  {\bf [C7]} are related to additional technicalities brought by the regression model, which force us to impose more
regularity than in the Gaussian sequence model.
\end{remark}

\begin{remark}
The result stated in Theorem~\ref{thm_1} is in the spirit of the previous work on the sharp asymptotics in minimax testing,
initiated by \citet{Ermakov90} in the problem of detection (\textit{i.e.}, $Q[f]=\|f\|_2^2$) under Gaussian white noise.
The explicit form\footnote{The first use of this type of weights for statistical purposes goes back to \citet{Pinsker80},
who showed that these weights lead to asymptotically minimax nonparametric estimators of the signal observed in Gaussian white noise.}
of the weights $\widehat w_{l,n}^*$ is obtained by solving a quadratic optimization problem called the extremal
problem in a series of recent works \citep{ingster2003nonparametric,ingster2009minimax,ingster2011estimation,Ingster_12}, see
also \cite{Ermakov04} for a similar result in the heteroscedastic GWNM. In the case $q_l=1$, $\forall l\in\mcL$, the
aforementioned extremal problem is equivalent to the saddle point problem (\ref{saddle}). In a nutshell, the main
differences of Theorem~\ref{thm_1} as compared to the existing results is the extension to the case of general coefficients $q_l$
and to non-Gaussian error distribution, as well as the use in the test statistic $U_n^\bw$ of the adjusted responses $\{\tilde x_i\}$
instead of the raw data $\{x_i\}$.
\end{remark}

\subsection{Lower bound}\label{ssec:2.4}

We shall state in this section the result showing that the rate $r_{n,\gamma}^*$ introduced in Theorem~\ref{thm_1}
is the minimax rate of testing and the exact separation constant associated with this rate is equal to one. This also
implies that the testing procedure proposed in previous subsection is not only minimax rate-optimal but also minimax
sharp-optimal among all possible testing procedures. In this subsection, we consider the functional classes
$\Sigma = \Sigma_{p,L}$ defined by
$$
\Sigma_{p,L}=\Big\{f=\sum_{l\in\mathcal{L}}\nolimits \theta_l[f]\varphi_l:\ \sum_{l\in\mathcal{L}}\nolimits c_l\theta_l[f]^2\leq 1,\
\|f\|_p\le L,\ \Pi_{S_F^c}f=0 \Big\}.
$$
Clearly, for $p>4$, this functional class is smaller than those satisfying conditions of Theorem~\ref{thm_1}. Therefore,
any lower bound proven for these functional classes will also be a lower bound for the functional classes for which Theorem~\ref{thm_1}
is applicable.
\begin{theorem}\label{thm_2}
Assume that $\xi_i$s are standard Gaussian random variables and that for every $T>0$, the set $\NL(T)=\{l\in S_F: c_l<Tq_l\}$ is finite.
For a prescribed significance level $\gamma\in(0,1)$, let $T_{n,\gamma}$ and $r_{n,\gamma}^*$ be as in Theorem~\ref{thm_1}.
If conditions {\bf [C1], [C3]} and
\begin{itemize}
\item[]{\bf {[C8]}} as $n\to\infty$, $|\mcN(T_{n,\gamma})|\to\infty$ so that $|\mcN(T_{n,\gamma})|\log(|\mcN(T_{n,\gamma})|)=o(n)$,
\item[]{\bf {[C9]}} as $n\to\infty$, $\max_{l\in\mcN(T_{n,\gamma})} c_l=o(n|\mcN(T_{n,\gamma})|^{1/2})$,
\end{itemize}
are fulfilled, then for every $C<1$ the minimax risk satisfies
\begin{equation}\label{eq:28}
\gamma_n(\mcF_0,\mcF_1(C r_{n,\gamma}^*)) \ge \gamma+o(1),\qquad\text{as}\qquad n\to\infty.
\end{equation}
\end{theorem}

Although the main steps of the proof of this theorem, postponed to the Appendix, are close to those of \citep{ingster2009minimax},
we have made several improvements which resulted in both shorter and more transparent proof and relaxed assumptions.
The most notable improvement is perhaps the fact that in condition {\bf [C3]} it is not necessary to have $C_3=1$. We
will further discuss this point and the other assumptions in the next section.

\begin{remark}\label{tauxseulementlower}
If we were only interested in minimax rate-optimality, we could have used simpler prior in the proof of Theorem~\ref{thm_2}
which would also yield the desired lower bound under slightly weaker assumptions. One can also deduce from the proof that
for a concrete pair $(\bc,\bq)$, a simple way to deduce the minimax rate of separation consists in finding
a sequence $r_n$ such that ${n(r_n)^2}\asymp M(r_n^{-2})^{1/2}$, where $M(T) = \sum_{l\in\mcN(T)} q_l^2$.
\end{remark}

\section{Examples}\label{sec:examples}
\subsection{Bases satisfying assumption {\bf[C3]}}\label{ssec:3.1}

First we give examples of orthonormal bases satisfying assumption {\bf [C3]}, irrespectively of the nature of arrays
$\bc$ and $\bq$ defining the smoothness class and the quadratic functional $Q$. One can take note that despite more general
settings considered in the present work, our assumption {\bf [C3]} is significantly weaker than the corresponding assumption in
\citep{ingster2009minimax}, which requires $C_3$ to be equal to one. In fact, in a remark, \cite{ingster2009minimax} suggest that
their proof remains valid under our assumption {\bf [C3]} if assumption {\bf [C4]} is strengthened to
$|\mcN(T_{n,\gamma})|=o(n^{2/3})$.  Due to a better analysis, we succeeded to establish sharp asymptotics under the
weak version of {\bf [C3]} without any additional price (except that a logarithmic factor appears now in the corresponding condition
in Theorem~\ref{thm_2}).

\paragraph{Fourier basis} Let us consider first the following Fourier basis in dimension $d$ for which $\mcL=\ZZ^d$ and
\begin{equation}\label{trig}
\varphi_\bk(\bt)=
\begin{cases}
1, & \bk=0 ,\\
\sqrt{2}\cos(2\pi\,\bk\cdot\bt), & \bk\in(\ZZ^d)_+,\\
\sqrt{2}\sin(2\pi\,\bk\cdot\bt), &-\bk\in(\ZZ^d)_+ ,
\end{cases}
\end{equation}
where $(\ZZ^d)_+$ denotes the set of all $\bk\in\ZZ^d\setminus\{0\}$ such that the first nonzero element of $\bk$ is positive and
$\bk\cdot\bt$ stands for the usual inner product in $\RR^d$. Since all the basis functions are bounded by $\sqrt{2}$, {\bf [C3]}
is obviously satisfied with $C_3=2$. Furthermore, if the set $\mcN(T)$ is symmetric, \textit{i.e.}, $\bk\in\mcN(T)$ implies $-\bk\in\mcN(T)$,
then  {\bf [C3]} is fulfilled with $C_3=1$.

\paragraph{Tensor product Fourier basis}
We can also consider the traditional tensor product Fourier basis as in \cite{ingster2009minimax}.
{\bf [C3]} is then obviously satisfied with $C_3=2^d$. Moreover,
if the set $\mcN(T)$ is orthosymmetric, \textit{i.e.}, $(k_1,\ldots,k_d)\in\mcN(T)$ implies $(\pm k_1,\ldots,\pm k_d)\in\mcN(T)$,
then  {\bf [C3]} is fulfilled with $C_3=1$.

\paragraph{Haar basis}
Let $\big\{\varphi_{j,k}(\cdot), j\in\NN, k\in\{1,\ldots,2^j\}\big\}$, be the standard orthonormal Haar basis on $[0,1]$, where $j$ is the
scale parameter and $k$ is the shift. The tensor product  $(\varphi_{\bj,\bk})_{\bj,\bk}$ Haar basis is then
$$
\varphi_{\bj,\bk}=\prod_{i=1}^d \varphi_{j_i,k_i},
$$
where $\bj=(j_1,\ldots,j_d)$ and $\bk=(k_1,\ldots,k_d)$. As shown in \citep{ingster2009minimax}, under the extra assumption that the coefficients $c_l=c_{\bj,\bk}$ and $q_l=q_{\bj,\bk}$ depend only on the scale parameter, \textit{i.e.}, $c_{\bj,\bk}=c_{\bj}$ and $q_{\bj,\bk}=q_{\bj}$,
assumption {[C3]} is satisfied with $C_3=1$. Note that the same holds true for the multivariate Haar basis defined in the more commonly used way
(see  \cite{cohen2003numerical},  chapter 2): $\big\{\varphi_l(\bt)=\prod_{i=1}^d \psi_{j,k_i}^{\omega_i}(t_i) \big\}$, where $l=(j,\bk,\oomega)$
such that $j\in\NN$, $\bk\in\{1,\ldots,2^j\}^d$ and  $\oomega\in\{0,1\}^d\setminus\{0\}$ with $\psi_{j,k}^{0}$ and $\psi_{j,k}^{1}$ being
 the scaled and shifted mother wavelet and father wavelet, respectively.

\paragraph{Compactly supported wavelet basis}
Since we are not limited to the case $C_3=1$, any orthonormal wavelet basis satisfies assumption {\bf [C3]}, as long as the wavelets are
compactly supported and provided that the coefficients $c_l$ and $q_l$ depend on the level of the resolution and not on the shift.

\subsection{Examples of estimators satisfying {\bf[C6]}}\label{ssec:3.15}

We present below pilot estimators that in two different contexts satisfy assumption \textbf{[C6]}.


\paragraph{Tensor-product Fourier basis}

For the first example, we assume that the orthonormal system $\{\varphi_l\}$ is the tensor product Fourier basis.
Then we have $\sup_l \sup_{t\in \Delta} |\varphi_l(t)|\leq 2^{d/2}$. The anisotropic Sobolev ball with
radius $R$ and smoothness $\ssigma=(\sigma_1,\ldots,\sigma_d)\in (0,\infty)^d$ is defined by
$$
W_{2}^{\ssigma} (R) = \Big\{f : \sum_{\bl\in\ZZ^d}\nolimits\sum_{i=1}^d\nolimits (2\pi l_i)^{2\sigma_i} \theta_\bl[f]^2\le R\Big\}.
$$
The estimator we suggest to use is constructed as follows. We first estimate $\theta_l[f]$ by
$\widehat\theta_l=\frac1n\sum_{i=1}^n x_i\varphi_l(\bt_i)$. Then we choose a tuning parameter
$T=T_n>0$ and define the pilot estimator
\begin{equation}\label{pilot}
\proj=\sum_{l\in S_F^c: c_l< T}\widehat{\theta}_l\varphi_l.
\end{equation}
To ease notation, we set $\mcN_1(T)=\{l\in S_F^c: c_l<T\}$ and $\mcN_2(T)=S_F^c\setminus \mcN_1(T)$.

\begin{lemma}\label{lem:3}
Assume that either one of the following conditions is satisfied:
\begin{itemize}
\item $\bc$ satisfies the condition $\sum_l c_l^{-1}<\infty$,
\item $\Sigma\subset W_2^\ssigma(R)$ for some $R>0$ and for some $\ssigma\in(0,\infty)^d$ such that
$\bar\sigma=(\frac1d\sum_i \frac1{\sigma_i})^{-1}> d/4$.
\end{itemize}
If $T=T_n\to\infty$ so
that $|\mcN_1(T)|=o(n^{1/2})$, then $\proj$ defined by (\ref{pilot}) satisfies {\bf [C6]}.
\end{lemma}

\paragraph{Compactly supported orthonormal wavelet basis} The same method can be applied in the case of an
orthonormal basis of compactly supported wavelets of $L_2[0,1]^d$. We suppose that the coefficients
$c_l=c_{j,\bk}$  correspond to those of a Besov ball $B_{2,2}^s$, \textit{i.e.}, $c_j=2^{js}$, and
that  $\sigma=s-d/4>0$. Let us set, for $J\in \NN$,
$$
\proj=\sum_{\bk\in [1,2^J]^d}\nolimits\widehat{\alpha}_{J,\bk}\varphi_{J,\bk}\quad\text{where}\quad
\widehat{\alpha}_{J,\bk}=\frac{1}{n}\sum_{i=1}^n\nolimits x_i\varphi_{J,\bk}(\bt_i).
$$
\begin{lemma}\label{lem:4}
If $J=J_n$ tends to infinity so that $2^{Jd}=o(n)$, then $\sup_{f\in\Sigma} E_f\|\pf-\proj\|^4\to 0$ as $n\to\infty$.
\end{lemma}

In the  following two subsections, we apply the previous results to two examples of quadratic functionals involving derivatives. The orthonormal system we use is the tensor product Fourier basis.

\subsection {Testing partial derivatives }\label{ssec:3.2}

We assume here that $f$ belongs to a Sobolev class with anisotropic constraints and
the quadratic functional $Q$ corresponds, roughly speaking, to the squared $L_2$-norm of a
partial derivative. More precisely, let $\aalpha\in\RR_+^d$ and $\ssigma\in\RR_+^d$ be two
given vectors and define, for every $\bl\in\mathcal{L}=\ZZ^d\setminus\{0\}$,
$$
q_\bl=\prod_{j=1}^{d}\nolimits(2\pi l_j)^{2\alpha_j},\qquad\text{and}\qquad
c_\bl=\sum_{j=1}^{d}\nolimits(2\pi l_j)^{2\sigma_j}.
$$
We will assume that $\sum_{j=1}^{d}(\alpha_j/\sigma_j)<1$.

For a function
$f=\sum_{l\in\mathcal{L}}\theta_l\varphi_l\in L_2(\Delta)$, we set $\|f\|_{2,c}^2=\sum_{l\in\mathcal{L}} c_l \theta_l^2$ and  $\|f\|_{2,q}^2=\sum_{l\in\mathcal{L}}  q_l \theta_l^2$.
Then, for a 1-periodic function which is differentiable enough, and if the $\alpha_j$ and $\sigma_j$ are integers, we have
$$
\|f\|_{2,q}^2=\|\partial^{\sum_j\alpha_j}f/\partial t_1^{\alpha_1}\ldots\partial t_{d}^{\alpha_{d}}\|_2^2,\qquad
\text{and}\qquad
\|f\|_{2,c}^2=\sum_{j=1}^d\nolimits\|\partial^{\sigma_j}f/\partial t_j^{\sigma_j}\|_2^2.
$$	

\begin{proposition}\label{premierex}
Let us define $\delta$, $\bar\sigma$, $(\kappa_j)$ and $\kappa$ by $\delta= \sum_{j=1}^{d}\alpha_j/\sigma_j$,
$\frac{1}{\bar\sigma}=\frac1d\sum_{j=1}^d \frac{1}{\sigma_j}$, $\kappa_j=\frac{1}{2\sigma_j}+\frac{\alpha_j}{\sigma_j}\frac{4\bar\sigma+d}{2\bar\sigma(1-\delta)}$
and $\kappa=\sum_{j=1}^d \kappa_j$. If $\delta<1$ and $\bar\sigma>d/4$, then the exact minimax rate $r_{n,\gamma}^*$ is given by $r_{n,\gamma}^*=C^*_\gamma r_n^*(1+o(1))$,
where the minimax rate $r_n^*$ and the exact separation constant are
$$
r_n^*= n^{-\frac{2\bar\sigma(1-\delta)}{4\bar\sigma+d}},\qquad\text{and}\qquad
C_\gamma^*=\big(4z^2_{1-\gamma/2} \kappa C(d, \ssigma, \aalpha)\big)^{\frac{\bar\sigma(1-\delta)}{4\bar\sigma+d}}(1+2\kappa^{-1})
^{\frac{2(1+\delta)\bar\sigma+d}{2(4\bar\sigma+d)}}
$$
with
$$
C(d, \ssigma, \aalpha)=\pi^{-d}\frac{\prod_{i=1}^{d}\Gamma(\kappa_i)}{ \big(\prod_{i=1}^{d}\sigma_i\big)(1-\delta)\Gamma(\kappa+2)}.
$$
Furthermore, the sequence of linear U-tests $\phi_n$ of Theorem \ref{thm_1} is asymptotically minimax with
$T_{n,\gamma}\sim (r_{n,\gamma}^*)^{-2}(1+2\kappa^{-1})$.
\end{proposition}

\begin{remark}\label{detection}
The previous result can be used for performing dimensionality reduction through variable selection \citep{CD11a}. Indeed, in a high-dimensional
set-up it is of central interest to eliminate the irrelevant covariates. The coordinate $t_i$ of $\bt$ is irrelevant if $f$ is constant
on the line $\{\bt\in\Delta: t_j=a_j \text{ for all } j\not=i\}$, whatever the vector $\ba\in\Delta$ is. This implies that the $i^{th}$
partial derivative of $f$ is zero. Therefore, one can test the relevance of a variable, say $t_1$, by comparing
$\|\partial f/\partial t_1\|_2$ with 0. In our notation, this amounts to testing hypotheses (\ref{eq:5}) with
$Q[f]=\|f\|_{2,\bq}^2$ such that $q_\bl = (2\pi l_1)^2$. Combining Proposition~\ref{premierex} and Theorem~\ref{thm_1},
one can easily deduce a minimax sharp-optimal test and the minimax sharp-rates for this variable selection problem.
\end{remark}

\begin{remark}
Another interesting particular case of the setting described in this subsection concerns the problem of component identification in
partial linear models \citep{Samarov}. We say that $f$ obeys a partial linear model if for some small subset $J$ of indices $\{1,\ldots,d\}$
and for a vector $\bbeta\in\RR^{|J^c|}$, one can write $f(\bt)= g(\bt_J)+ \bbeta\!^\top\bt_{J^c}$ for every $\bt\in\Delta$. The problem of
component identification in this model is to determine for an index $j$ whether $j\in J$ or not. One way of addressing this issue is
to perform a test of hypothesis $Q[f]=\|f\|_{2,\bq}^2=0$, where $q_\bl=(2\pi l_j)^4$. Roughly speaking, this corresponds to checking
whether the second order partial derivative of $f$ with respect to $t_j$ is zero or not (if the null is not rejected, then $j\in J^c$).
Once again, Proposition~\ref{premierex} and Theorem~\ref{thm_1} provide a minimax sharp-optimal test for this problem along with the
minimax rates and exact separation constants.
\end{remark}

\begin{remark}
In the case where the covariates $\bt_i$ are not observable and only $x_i$'s are available, our model coincides with the convolution model,
for which the minimax rates of testing were obtained by \cite{Butucea07} in the one-dimensional case with simple null hypothesis. It would be
interesting to extend our results to such a model and to get minimax rates and, if possible, separation constants in the multidimensional
convolution model.
\end{remark}

\subsection{Testing the relevance of a direction in a single-index model}\label{ssec:3.3}

Recall that a single-index model is a particular case of (\ref{eq:1}) corresponding to functions $f$ that can be written in the form
$f(\bt)=g(\bbeta_0^\top\bt)$ for some univariate function $g:\RR\to\RR$ and some vector $\bbeta_0\in\RR^d$. Assume now that for a candidate
vector $\bbeta\in\RR^d\setminus\{\boldsymbol{0}\}$ we wish to test the goodness-of-fit of the single-index model \citep{DJS, GL}. This corresponds to testing the hypothesis
$$
\exists g:\RR\to\RR\qquad\text{such that}\qquad f(\bt)=g(\bbeta^\top\bt),\quad \forall \bt\in\Delta.
$$
This condition implies that
$\frac{\partial f}{\partial t_i}(\bt)\equiv  \frac{\beta_i}{\|\bbeta\|_2^2}\sum_{j=1}^d \beta_j\frac{\partial f}{\partial t_j}(\bt)=
\frac{\beta_i}{\|\bbeta\|_2^2} \;\bbeta^\top\nabla f(\bt)$, $\forall i\in\{1,\ldots,d\},$
which in turn can be written as
$$
\sum_{i=1}^d \Big( \frac{\partial f}{\partial t_i}-\frac{\beta_i}{\|\bbeta\|_2^2}\;\bbeta^\top\nabla f(\bt) \Big)^2\equiv 0.
$$
Without loss of generality, we assume that $\|\bbeta\|_2=1$ and set
$q_{\bl}= \sum_{i=1}^d (2\pi)^2\big(l_i- (\bbeta^\top \bl) \beta_i\big)^2=(2\pi)^2\big(\|\bl\|_2^2-(\beta^\top\bl)^2\big).$
We consider homogeneous Sobolev smoothness classes, that is $ c_\bl=  \sum_{i=1}^d (2\pi l_i)^{2\sigma}$, with $\sigma>d/4$.
Then, when $\sigma$ is an integer,  for a 1-periodic function which is smooth enough,
$$
\|f\|_{2,\bc}^2=\sum_{i=1}^d\Big\|\frac{\partial^{\sigma}f}{\partial t_i^{\sigma}}\Big\|_2^2\quad\text{ and }
\quad\|f\|_{2,\bq}^2= \sum_{i=1}^d \Big\| \frac{\partial f}{\partial t_i}- \beta_i[\bbeta^\top\nabla f ] \Big\|^2 .
$$
To state the result providing the minimax rate and the exact constant in this problem, we introduce the constants
\begin{align*}
\bar C_0&=\frac1{(2\pi)^d}\int_{\RR^d}\big[\|\bx\|_2^2-(\bbeta^\top\bx)^2-\|\bx\|_{2\sigma}^{2\sigma}\big]_+^2d\bx,\\
\bar C_1&=\frac1{(2\pi)^d}\int_{\RR^d}\big(\|\bx\|_2^2-(\bbeta^\top\bx)^2\big)\big(\|\bx\|_2^2-(\bbeta^\top\bx)^2-\|\bx\|_{2\sigma}^{2\sigma}\big)_+d\bx,
\end{align*}
and $\bar C_2=\bar C_1-\bar C_0$.

\begin{proposition}\label{deuxiemeex}
In the setting described above, the exact minimax rate $r_{n,\gamma}^*$  is given by $r_{n,\gamma}^*=C^*_\gamma r_n^*(1+o(1))$, where
$$
r_n^*=n^{-\frac{2(\sigma-1)}{4\sigma+d}}
\qquad\text{and}\qquad
C_\gamma^*=\Big(\frac{4z_{1-\gamma/2}(\bar C_1/\bar C_2)^{\frac{d+4}{2(\sigma-1)}}\bar C_1^2}{ \sigma^{d-1}(\sigma-1)\bar C_0 }\Big)^{\frac{\sigma-1}{4\sigma+d}}.
$$
The sequence of tests $\phi_n$ of Theorem \ref{thm_1} is minimax sharp-optimal if $T=T_{n,\gamma}$ is chosen as
$T= (C_\gamma^* r_n^*)^{-2}\big({\bar C_1}/{\bar C_2}\big)$.
\end{proposition}

\begin{remark}
The testing procedures provided in Propositions \ref{premierex} and \ref{deuxiemeex} require the precise knowledge
of the smoothness parameter $\ssigma$, which may not be available in practice. Indeed, the parameter $\ssigma$ explicitly
enters in the definition of the tuning parameter $T_n$. The adaptation to the unknown smoothness $\ssigma$ is an interesting
problem for future research. We believe that rates of separation similar to those of Propositions
\ref{premierex} and \ref{deuxiemeex} can be established for adaptive tests (up to logarithmic factors) using the Berry-Esseen
type theorem for degenerate $U$-statistics of \citet{But_et_al2009}.
\end{remark}

\section{ Nonpositive and nonnegative diagonal quadratic functionals }\label{sec:NPF}

In this section we consider the more general setting obtained by abandoning the assumption that all the entries $q_l$ of the array
$\bq$ have the same sign. That is, we still have $Q[f]=\sum_{l\in\mathcal{L}}\nolimits q_l\theta_l^2,$
but now
\begin{align}\label{mcl+-}
\mcL_+=\{l : q_l>0\}\neq \varnothing\qquad \text{and}\qquad\mcL_-=\{l : q_l<0\}\neq \varnothing.
\end{align}
The sets $\mcF_0$ and  $\mcF_1(r_n)$ are defined as before, cf.\ (\ref{eq:5}),
and we use the same notation as in the positive case. Namely, for $T>0$, we set
$\NL(T)=\big\{l\in {S_F} : c_l<T|q_l|\big\}$,  $ N(T)=|\NL(T)|$ and  $M(T)=\sum_{l\in \NL(T)}q_l^2$.

We point out that, in the case considered in this section, a phenomenon of phase transition occurs: there is a regular case in
which the rate is independent of the precise degree of smoothness, and an irregular case where the rate
is smoothness-dependent. To be more precise, let ${|Q|}$ denote the diagonal positive quadratic functional whose coefficients
are $|q_l|$ for every $l\in \mathcal{L}$.  Let us recall that the minimax rate $r_n^*$ in testing the significance of $|Q|[f]$
(see Remark \ref{tauxseulementlower}) is determined by
$$
{n(r_n^*)^2}\asymp {M({r_n^*}^{-2})^{1/2}}.
$$
In our context, this rate corresponds to the irregular case: if $\Sigma$ contains functions that are not smooth enough
(compared to the difficulty of the problem, that is to say if $q_{l}$'s are ``too large" compared
to $c_{l}$'s), the minimax rate corresponding to $Q$  is the same as for ${|Q|}$ obtained in previous sections.
By contrast, in the regular case, the minimax rate is smoothness-independent and equals $r_n^*=n^{-1/4}$.

\subsection{Testing procedure and upper bound on the minimax rate}

The testing procedure we use in the present context is of the same type as the one used for nonnegative quadratic functionals. More precisely,
for a tuning parameter  $T_n$ and for a threshold $u$, we set $\phi_n(T)=\1_{|U_n(T)|>u}$, where  the
$U$-statistic $U_n(T)$ is defined by
$$
U_n(T)=\binom{n}{2}^{-1/2}\sum_{1\leq i<j \leq n} x_ix_j G_T(\bt_i, \bt_j).
$$
with $G_T(\bt_1,\bt_2)=M(T)^{-1/2}\sum_{l\in\NL(T)} q_l\varphi_l(\bt_1)\varphi_l(\bt_2)$.

\begin{theorem}\label{nonpositiveupper}
Let $\gamma\in(0,1)$ be a fixed significance level. Let us denote by $\ttT_Q[f]$ the linear functional $\ttT_Q[f]=\sum_{l\in\mcN(T)} q_l \theta_l[f]\varphi_l$.
Assume that  $T>0$ is such that the assumptions
\begin{itemize}
\item[] {\bf [D1]} there exists $D_1>0$ such that $|\mcN(T)| \max_{l\in\mcN(T)} q_l^2\le D_1 \sum_{l\in\mcN(T)} q_l^2$,\\[-10pt]
\item[] {\bf [D2]} there exists $D_2>0$ such that $\sup_{\bt\in\Delta} \sum_{l\in\mcN(T)} \varphi_l(\bt)^2 \le D_2|\mcN(T)|$,\\[-10pt]
\item[] {\bf [D3]} there exists $D_3>0$ such that $\sup_{f\in\Sigma} \|f\|_4\le D_3$,\\[-10pt]
\item[] {\bf [D4]} there exists $D_4>0$ such that $\sup_{f\in\Sigma} \|f\cdot \ttT_Q[f]\|_2\le D_4$,\\[-10pt]
\end{itemize}
are fulfilled. Set $B_1=6+12D_1D_2D_3^2+6D_1D_2D_3^4$ and $B_2=4D_4$. Then, for every
$$
u \ge \frac{n}{T\sqrt{2M(T)}}+\gamma^{-1/2}\big(B_1+B_2 nM(T)^{-1}\big)^{1/2},
$$
the type I error is bounded by $\gamma/2$: $\sup_{f\in\mcF_0}P_f(\phi_n(T)=1) \leq  \frac{\gamma}{2}$.\\
If, in addition,
$$
\rho^2\ge \big[u+\gamma^{-1/2}\big(B_1+B_2 nM(T)^{-1}\big)^{1/2}\big]\frac{\sqrt{2M(T)}}{n} +\frac{1}{T}
$$
then the type II error is also bounded by $\gamma/2$: $\sup_{f\in\mcF_1(\rho)}P_f(\phi_n(T)=0) \leq  \frac{\gamma}{2}$.\\
As a consequence, if we choose $u=(2M(T))^{-1/2}({n}/{T})+\gamma^{-1/2}\big(B_1+B_2 nM(T)^{-1}\big)^{1/2}$ then
the cumulative error rate of the test $\phi_n(T)$ is bounded by $\gamma$ for every alternative $\mcF_1(\rho)$
such that $\rho^2\ge 2\sqrt2\gamma^{-1/2}n^{-1}\big(B_1M(T)+B_2 n\big)^{1/2} +{2}T^{-1}$.
\end{theorem}

This theorem provides a nonasymptotic evaluation of the cumulative error rate of the linear U-test based on the
array $w_l\propto q_l$ truncated at the level $T$. In the cases where the constants $B_1$ and $B_2$ can be reliably
estimated and the function $M(T)$ admits a simple form, it is reasonable to choose the truncation level $T$ by
minimizing the expression $2\sqrt2\gamma^{-1/2}n^{-1}\big(B_1M(T)+B_2 n\big)^{1/2} +{2}T^{-1}$. By choosing $T$
in such a way, we try to enlarge the set of alternatives for which the cumulative error rate stays below the prescribed
level $\gamma$. Therefore, the last theorem implies the following non-asymptotic upper bound on the minimax
rate of separation:
\begin{align}\label{nonas}
(r_{n,\gamma}^*)^2\le \inf_{T>0}\Big(\frac{2\sqrt2\big(B_1M(T)+B_2 n\big)^{1/2}}{n\gamma^{1/2}} +\frac{2}{T}\Big).
\end{align}
This non-asymptotic bound clearly shows the presence of two asymptotic regimes. The first one corresponds
to the case where $n$  is much larger than $M(T^*)$, whereas the second regime corresponds to $n=o(M(T^*))$. Here, $T^*$
is the minimizer of the bound on $\rho^2$ obtained in the theorem above. The next corollary exhibits the rates of
separation in these two different regimes.

\begin{corollary}\label{cor}
Assume that the arrays $\bq$ and $\bc$ are such that $M(\alpha T)\asymp_{T\to\infty} M(T)$ for every $\alpha>0$.
Let $T_n^0$ be any sequence of positive numbers satisfying $T_n^0\sqrt{M(T_n^0)}\asymp n$.  If for the sequence
$T_n=T_n^0\wedge n^{1/2}$ all the assumptions  of Theorem~\ref{nonpositiveupper} are satisfied, then for some $C>0$
the linear U-test $\phi_n(T)$ based on the threshold $T=T_n$ satisfies
$$
\gamma_n(\mcF_0,\mcF_1(CT_n^{-1/2}), \phi_n)\le \gamma.
$$
Thus, the rate of convergence is $r_n^*=(T_n^0)^{-1/2}$ if $T^0_n=o(n^{1/2})$ and $r_n^*=n^{-1/4}$ otherwise.
\end{corollary}

\begin{remark}
Condition \textbf{[D4]} of Theorem~\ref{nonpositiveupper} is more obscure than the other assumptions of theorem.
Clearly, it imposes additional smoothness constraints on the function $f$. Using the Cauchy-Schwarz inequality,
one can easily check that either one of the assumptions \textbf{[D4-1]} and \textbf{[D4-2]} below is sufficient
for \textbf{[D4]}:
\begin{itemize}
\item[] \textbf{[D4-1]} For some constants $D_5$ and $D_6$, $\sup_{f\in\Sigma}\|f\|_\infty\le D_5$ and $\max_{l\in\mcN(T)} |q_l/c_l|\le D_6$.
\item[] \textbf{[D4-2]} For some constant $D_4'$, $\sup_{f\in\Sigma} \|\ttT_Q[f]\|_4\le D_4'$.
\end{itemize}
\end{remark}

\subsection{Lower bound on the minimax rate}

We will show in this subsection that the asymptotic rate of separation provided by Corollary~\ref{cor} is unimprovable, in the sense
that there is no testing procedure having a faster separation rate. To this end, for every $a\in\{-,+\}$ we set
$M_a(T)= \sum_{l\in \mcL_a\cap \mcN(T)} q_l^2$, $N_a(T)=|\mcL_a\cap \mcN(T)|$,
$$
M^*(T)=M_+(T)\vee M_-(T),\qquad N^*(T)=N_+(T)\1_{\{M_+(T)>M_-(T)\}}+N_-(T)\1_{\{M_+(T)\le M_-(T)\}}.
$$

\begin{theorem}\label{nonpositivelower} Let us consider the problem of testing $H_0:f\in\mcF_0$ against
$H_1:f\in\mcF_1(\rho)$, where $\mcF_0$ and $\mcF_1$ are defined by (\ref{eq:5}) and
$$
\Sigma_{L}=\Big\{f=\sum_{l\in\mathcal{L}}\nolimits \theta_l[f]\varphi_l:\ \sum_{l\in\mathcal{L}}\nolimits c_l\theta_l[f]^2\leq 1,\
\|f\|_4\vee \|f\cdot\ttT_Q[f]\|_2\le L\Big\}.
$$
Assume that  the sets $\mcL_+$ and $\mcL_-$ defined by (\ref{mcl+-}) are both nonempty and that $\xi_i$'s are Gaussian.
The following assertions are true.
\begin{enumerate}
\item  For every $\gamma<1/4$ there exists $C>0$ such that $\liminf_{n\to\infty} \gamma_n(\mcF_0, \mcF_1(Cn^{-1/4}))>\gamma$.
\item  Let $T^0_n$ be a sequence of reals such that $4T^0_n\sqrt{M(T^0_n)}\ge n z_{1-\gamma/2}^{-1}$ as $n\to\infty$. If
the assumptions \textbf{\bf [D1]} (cf.\ Theorem~\ref{nonpositiveupper})  and
\begin{itemize}
\item[] {\bf [D5]} $N^*(T^0_n)\to\infty$ so that $N^*(T^0_n)\log N^*(T^0_n) =o(n)$,
\item[] {\bf [D6]} there exists $D_6>0$ such that $\sup_{\bt\in\Delta} \sum_{l\in\mcN^*(T_n^0)} \varphi_l(\bt)^2 \le D_6 N^*(T_n^0)$,
\end{itemize}
are fulfilled, then there exists $C>0$ such that $\liminf_{n\to\infty} \gamma_n\big(\mcF_0, \mcF_1\big(C{(T^0_n)}^{-1/2}\big)\big)\ge \gamma.$
\end{enumerate}
\end{theorem}

\begin{corollary}
Combining the two assertions of this theorem, we get that the minimax rate of separation $r_n^*$ is lower bounded by
$n^{-1/4}\vee (T_n^0)^{-1/2}= (n^{1/2}\wedge T_n^0)^{-1/2}=T_n^{-1/2}$. Thus, if the conditions of Theorems~\ref{nonpositiveupper}
and \ref{nonpositivelower} are satisfied, then the minimax rate of separation is given by $r_n^*=T_n^{-1/2}$, where
$T_n=n^{1/2}\wedge T_n^0$ and $T_n^0$ is determined from the relation $T^0_nM(T^0_n)^{1/2}\asymp n$.
\end{corollary}

\subsection{Testing equality of norms}\label{ssec:4.3}

As an application of the testing methodology developed in this section, we consider the problem of testing
the equality of norms of two functions observed in noisy environment.  More precisely, let us consider the
following two-sample problem: for $i=1,\ldots,n$ we observe $(x_{1,i},\bt_{1,i})$ and $(x_{2,i},\bt_{2,i})$
such that
$$
x_{s,i}=g_s(\bt_{s,i})+\xi_{s,i}, \qquad i=1,\ldots, n; \quad s=1,2,
$$
where $\bt_{s,i}$'s are independent random vectors drawn from the uniform distribution over $[0,1]^d$. Furthermore, we assume that
$\xi_{s,i}$'s are i.i.d.\ such that $E(\xi_{s,i}|\{\bt_{s,j}\})=0$, $E(\xi_{s,i}^2|\{\bt_{s,j}\})=1$ and, for some $C_\xi<\infty$,
$E(\xi_{s,i}^4|\{\bt_{s,j}\})\le C_\xi$ almost surely.

Assuming that both $g_1$ and $g_2$ belong to a smoothness class $\Sigma$, we wish to test the hypothesis
$$
H_0: \|g_1\|_{W_2^\aalpha}=\|g_2\|_{W_2^\aalpha},\qquad \text{against}\qquad H_1: \big|\|g_1\|_{W_2^\aalpha}^2-\|g_2\|_{W_2^\aalpha}^2\big|\ge \rho^2,
$$
where for any function $g$ we denoted by $\|g\|_{W_2^\aalpha}$ the (anisotropic) Sobolev norm of order $\aalpha\in\RR_+^d$ (the precise definition is 
given below). It can be useful to perform such a test prior to using a shifted curve model in the context of curve registration \citep{DalCol12,Collier12}. Indeed,
if there exists $\ttau\in[0,1]^d$ such that $g_1(\bt)=g_2(\bt-\ttau)$ for every $\bt\in[0,1]^d$ and the function
$g_1$ is one-periodic, then necessarily $\|g_1\|_{W_2^\aalpha}=\|g_2\|_{W_2^\aalpha}$ for any $\aalpha$. Thus, the rejection of the null hypothesis
implies the inadequacy of the shifted curve model. In order to show how this type of test can be derived from
the framework presented in the previous subsections, let us consider  the case of a Sobolev ellipsoid $\Sigma$.

Let $\{\psi_m\}_{l\in\mcM}$ be an orthonormal basis of the subspace $L_{2,c}([0,1]^d)$ of $L_2([0,1]^d)$ consisting of all the functions
orthogonal to the constant function. We will assume that both $g_1$ and $g_2$ are centered (this implies that they are
orthogonal to the constant function as well). The Fourier coefficients of a function $g$ w.r.t.\ a basis $\{\psi_m\}$ will be denoted
by $\theta_m^\psi[g]$. We assume that for some array $\bc$ and some constant $L>0$ it holds that
$$
g_s\in\Sigma^0_L=\Big\{g\in L_{2,c}([0,1]^d): \sum_{m\in\mcM}\nolimits c_m\theta^\psi_m[g]^2\le 1,\ \|g\|_4\le L\Big\},\qquad\forall s\in\{1,2\}.
$$
Assume now that we wish to test
$$
H_0: \sum_{m\in\mcM} q_m\theta_m^\psi[g_1]^2=\sum_{m\in\mcM} q_m\theta^\psi_m[g_2]^2,
\quad \text{against}\quad H_1: \bigg|\sum_{m\in\mcM} q_m(\theta_m^\psi[g_1]^2-\theta_m^\psi[g_2]^2)\bigg|\ge \rho^2,
$$
where $\bq=\{q_m\}$ is a given array.
In order to show that this problem can be solved within the framework of the previous subsections, we introduce the functional set
$$
\Sigma_L = \Big\{f:[0,1]^{2d}\to\RR : f(t_1,\ldots,t_{2d}) = g_1(t_1,\ldots,t_{d})+g_2(t_{d+1},\ldots,t_{2d}) \text{ with } g_1,g_2\in\Sigma^0_{L}\Big\}.
$$
Setting $\mcL=\mcM\times\{1,2\}$ and for $l=(m,s)\in\mcM\times\{1,2\}$
$$
\varphi_l(\bt_1,\bt_2)=\psi_m(\bt_s),\qquad \text{for all}\qquad \bt=(\bt_1,\bt_2)\in[0,1]^d\times [0,1]^d,
$$
we get an orthonormal basis of $\Sigma_L$. Clearly, for a function $f\in\Sigma_L$, we have $\theta_l^\varphi[f]=\theta_{m,s}^\varphi[f]=\theta^\psi_{m}[g_s]$.
This implies that $\Sigma_L$ is included in the set $\Sigma^2_L=\{f:\sum_{(m,s)} c_m\theta_{m,s}^\varphi[f]^2\le 2;\ \|f\|_4\le2L\}$ and contains the set
$\Sigma^1_L=\{f:\sum_{(m,s)} c_m\theta_{m,s}^\varphi[f]^2\le 1, \|f\|_4 \le L\}$. Therefore, for studying the rate of separation of a testing procedure we
can assume that $f\in\Sigma^2_L$, whereas for establishing lower bounds on the minimax rate of separation we can use the relation $\Sigma^1_L\subset\Sigma_L$.
In both cases, this perfectly matches the framework of the previous subsections.

We give a concrete example by setting $\mcM=\ZZ^d$ and choosing as $\{\psi_\bm\}$ the Fourier basis in dimension $d$. Similarly to
the example in Subsection \ref{ssec:3.2}, we focus on anisotropic Sobolev smoothness classes defined via coefficients
$$
c_{\bm}= \sum_{j=1}^{d}(2\pi m_j)^{2\sigma_j},\qquad \bm\in\ZZ^d,
$$
for some $\ssigma=(\sigma_1,\ldots,\sigma_d)\in\RR^d_+$. As it was done previously, $\delta=\sum_{j=1}^d \alpha_j/\sigma_j$ and 
$\bar\sigma$ stands for the harmonic mean of $\sigma_j$'s: $\bar\sigma=\big(\frac1d\sum_{j=1}^d \sigma_j^{-1}\big)^{-1}$. We still 
assume that $\delta<1$ and $\bar\sigma>d/4$. To test the equality of Sobolev norms, we introduce the coefficients $q_\bl$, $\bl=(\bm,s)\in\ZZ^d\times\{1,2\}$,
of the quadratic functional $Q$:
$$
q_{\bm,s}=(-1)^s \prod_{j=1}^d(2\pi m_j)^{2\alpha_j},\qquad (\bm,s)\in\ZZ^d\times\{1,2\}.
$$
Theorems~\ref{nonpositiveupper} and \ref{nonpositivelower}, as well as the computations done in the proof of Proposition~\ref{premierex},
imply that the minimax rate of separation in the problem described above  is: $r_n^*=n^{-\frac{2(1-\delta)\bar\sigma}{4\bar\sigma+d}\wedge \frac14}$.
It is interesting to note that if $\delta\ge 1/2$ then we are in the irregular regime irrespectively of the value of $\bar\sigma$ and, therefore, the
rate of separation is strictly slower than the rate $n^{-1/4}$.

\section{Conclusion and outlook}\label{sec:conc}

We have presented a statistical analysis of the problem of testing the significance of the value $Q[f]$ for a quadratic functional
$Q$ of a regression function $f$. While the overwhelming majority of previous research focused on the case of a function $f$ observed at any
point in Gaussian white noise, we have considered here the more realistic setting when the observations are noisy values of $f$ at a finite number
of points uniformly randomly drawn from $[0,1]^d$. Furthermore, we have explored not only the case of positive semi-definite functional $Q$ but also
the situation when $Q$ is neither positive nor negative semi-definite. In the first situation we have established asymptotic results providing
the minimax rates of separation along with the sharp constants. In the second case, the analysis we have carried out is nonasymptotic and leads
to the asymptotically minimax rate of separation, which exhibits two different regimes: the regular and the irregular regimes. Another distinctive
feature of our approach is that we have put the emphasis on the multidimensional setting $d>1$, even if at this stage we have not tackled the
problem of increasingly high dimensionality: $d=d_n\to\infty$ as the sample size $n$ tends to infinity.

The results we have obtained are closely related to those of estimating quadratic functionals. While the presence of such a relation is not
surprising in itself, the actual nature of the relation uncovers some interesting new phenomena. In fact, the test statistic used in our work
is a properly normalized estimator of the quadratic functional $Q[f]$, which is constructed following the classical approach of weighted squared
linear functional estimation (cf., for instance, \cite{DonohoNussbaum}). Usually, the proper choice of the shrinkage weights and the resulting rates of
convergence differ in the problem of hypothesis testing and in the problem of estimation. This is why the well-known ``elbow'' effect (phase
transition) in estimating quadratic functionals disappears when the problem of hypotheses testing is considered for $Q[f]=\int_{[0,1]} f^2(t)\,dt$.
Interestingly, the results of Section~\ref{sec:NPF} show that this difference between the rates of convergence is erased when the quadratic
functional $Q[f]$ is neither positive nor negative. In fact, the rates of separation we have obtained in this case coincide with the
square-root of the rates of estimation \citep{DonohoNussbaum, Fan91}. Therefore, the ``elbow'' effect is present in this problem of hypotheses
testing. More interestingly,
the rates of separation we obtained in the case of positive semi-definite functionals $Q$ coincide with the rates of estimation of the functional
$\sqrt{Q[f]}$ in the case $Q[f] = \int_{[0,1]} f^2(t)\,dt$, at least in the Gaussian white noise model \citep{LepNemSpok}. An intriguing question
worth of being further explored is whether this analogy extends to the model of regression with random design and general positive semi-definite
functionals $Q[f]$.

Several relevant problems remained out of scope of the present paper. Most important ones are the possibility of extending our results to
the case of nondiagonal functionals $Q[f]$ and the attainability of the obtained rates of separation by adaptive tests. More specifically,
in some applications such as in deconvolution it may be more realistic to assume that the functional basis in which the smoothness of $f$ is
expressed does not coincide with the basis of the singular vectors of (the bilinear operator underlying) $Q$. This means that $Q[f]$
will be of the form $Q[f]=\sum_{l,l'\in\mcL} q_{l,l'} \theta_l[f]\theta_{l'}[f]$ rather than $Q[f]=\sum_{l\in\mcL} q_{l} \theta_l[f]^2$.
Furthermore, it would be more reasonable to replace the assumption $\sum_{l\in\mcL} c_l \theta_l[f]^2\le 1$ with some \textit{known}
array $\bc=\{c_l\}_{l\in\mcL}$ by the assumption $\sum_{l\in\mcL} c_l(\mu^*) \theta_l[f]^2\le 1$, where  $\bc(\mu)= \{c_l(\mu)\}_{l\in\mcL}$
is a collection of arrays such that the mapping $\mu\mapsto \bc(\mu)$ is known but the precise value $\mu^*$ for which the smoothness constraint
is valid is unknown. In the light of the previous discussion, it seems natural to study these two extensions (nondiagonal $Q$ and adaptation to
the smoothness class) by considering the problem of testing and the problem of estimating functionals in a joint framework. In particular, any
progress in establishing upper bounds for estimators of $Q[f]$ or $\sqrt{|Q[f]|}$ will straightforwardly lead to upper bounds for the rates of
separation. Quite surprisingly, these problems of estimation received little attention in the context of nonparametric regression\footnote{
Minimax and adaptive estimation for (nondiagonal) quadratic functionals is well studied in the case of Gaussian white noise model. However,
these results do not always carry over the regression model as noticed by \citet{Efrom2003}.}. They constitute
interesting avenues for future research.

\appendix

\section{Proofs of results stated in Section~\ref{sec:2}}

\subsection{Proof of Proposition~\ref{prop_1}}

Throughout the proof, the terms  $o(1)$, $O(1)$  and the equivalences are uniform over $\Sigma$. Let $\mcL(\bw_n)$ be the support of $\bw_n$.
$E_f^{\mcD_2}$ will denote the conditional expectation with respect to $\mcD_2$.
We define
\begin{align}
&h_n[f,\bw_n] = \Big(\frac{m(m-1)}{2}\Big)^{1/2}\sum_{l\in \mcL(\bw_n)}\nolimits w_{l,n} \theta_l^2[f],\label{h_n}\\
&G_n(\bt_1,\bt_2)=\sum_{l\in \mcL(\bw_n)}w_{l,n}\varphi_l(\bt_1)\varphi_l(\bt_2).
\end{align}
This allows us to rewrite the U-statistic $U_n$ in the form $U_n=U_{n,0}+U_{n,1}+U_{n,2}$
where
$$
U_{n,k}=\Big(\frac{2}{m(m-1)}\Big)^{1/2}\sum_{1\leq i<j\leq m}K_{n,k}(\tilde{\bz}_i,\tilde{\bz}_j),\qquad k=0,1,2,
$$
are U-statistics with the kernels
\begin{align}
K_{n,0}(\tilde\bz_1,\tilde\bz_2) &= \xi_1\xi_2G_n(\bt_1,\bt_2),\\
K_{n,1}(\tilde\bz_1,\tilde\bz_2) &= \Big[\xi_1\big(f-{\proj}\big)(\bt_2)+\xi_2\big(f-{\proj}\big)(\bt_1)\Big]G_n(\bt_1,\bt_2),\\
K_{n,2}(\tilde\bz_1,\tilde\bz_2) &= \big(f-{\proj}\big)(\bt_1)\big(f-{\proj}\big)(\bt_2)G_n(\bt_1,\bt_2).
\end{align}
To prove Proposition~\ref{prop_1} and the subsequent results, we need two auxiliary lemmas.

\begin{lemma}\label{lem:AD1}
Let  $\bw_n=(w_{l,n})_{l\in\mcL}$ be a family of positive numbers containing only a finite number of nonzero entries and such that
$\sum_{l\in\mcL} w_{l,n}^2=1$. Let $\mcL(\bw_n)$ be the support of $\bw_n$. Then the expectation of the U-statistic $U_{n}$ is given by:
\begin{align*}
E_f[U_{n}]=E_f[U_{n,2}]=h_n[f,\bw_n],
\end{align*}
whereas for the variances it holds
\begin{align}
E_f[U_{n,0}^2]
	&=1,\nonumber\\
E_f[U_{n,1}^2]
	&\le 2\|\bw_n\|_\infty^2\Big(\sup_{\bt\in\Delta}\sum_{l\in\mcL(\bw_n)}\nolimits \varphi_l^2(\bt)\Big)
		   \Big(\|\Pi_{S_F}f\|_2^2+E_f\big[\|\pf-\proj\|_2^2\big]\Big),\label{U1}\\
Var_f[U_{n,2}]
	&\leq 8\|\bw_n\|_\infty^2\Big(\sup_{\bt\in\Delta}\sum_{l\in\mcL(\bw_n)} \varphi_l^2(\bt)\Big)
		   \Big(\|\Pi_{S_F}f\|_4^4+E_f\big[\|\pf-\proj\|_4^4\big]\Big)\nonumber\\
	&\quad+8h_n[f,\bw_n]\|\bw_n\|_\infty\Big(\sup_{\bt\in\Delta}\sum_{l\in\mcL(\bw_n)} \varphi_l^2(\bt)\Big)^{1/2}
		   \Big(\|\Pi_{S_F}f\|_4^4+E_f\big[\|\pf-\proj\|_4^4\big]\Big)^{1/2}.\label{U2}	
\end{align}
\end{lemma}
\begin{proof}
It is clear that $E_fU_{n,0}=E_fU_{n,1}=0$, while
$$
E_f[U_{n,2}]
	= \Big(\frac{m(m-1)}{2}\Big)^{1/2} E_f[K_{n,2}(\tilde\bz_1,\tilde\bz_2)]
$$
with
\begin{align*}
E_f[K_{n,2}(\tilde\bz_1,\tilde\bz_2)]&=E_{f}\Big[\sum_{l\in \mcL(\bw_n)}\nolimits w_{l,n}
\Big( \int \big(f(\bt)-\proj(\bt)\big)\varphi_l(\bt)d\bt\Big)^2\Big].
\end{align*}
As $\proj \in \text{ span}\big(\{\varphi_l\}_{l\in {S_F^c}}\big)$, we have  $\int {\proj}\varphi_l=0$ for all $ l\in {S_F}$.
Therefore
$$
E_f[U_{n,2}]=\Big(\frac{m(m-1)}{2}\Big)^{1/2}\sum_{l\in \mcL(\bw_n)}\nolimits w_{l,n} \theta_l^2[f]=h_n[f,\bw_n].
$$
Now, let us evaluate the variances. Since $\xi_i$s are non correlated zero-mean random variables with variance one, and
$\varphi_l$'s are orthonormal, it holds that $E_f[U_{n,0}^2]=E_f[G_n(\bt_1,\bt_2)^2]=\sum_l w_{l,n}^2=1$. For $U_{n,1}$, we have
\begin{align*}
\text{Var}_f [U_{n,1}]=E_f[ U_{n,1}^2]&
=E_{f}E_f^{\mcD_2}[K_{n,1}^2(\tilde\bz_1,\tilde\bz_2)].
\end{align*}
Using the definition of $G_n(\bt_1,\bt_2)$, we get
\begin{align*}
E_f^{\mcD_2}[K_{n,1}^2(\tilde\bz_1,\tilde\bz_2)]&=2\int_\Delta\int_\Delta \big(f-\proj\big)^2(\bt_1)G_n^2(\bt_1,\bt_2)d\bt_1d\bt_2\\
&=2\int_\Delta \big(f-\proj\big)^2(\bt_1)\sum_{l\in \mcL(\bw_n)} w_{l,n}^2\varphi_l^2(\bt_1)d\bt_1\\
&\le 2\Big(\max_{l\in\mcL(\bw_n)} w_{l,n}^2\Big)\Big(\sup_{\bt\in\Delta}\sum_{l\in\mcL(\bw_n)}\nolimits \varphi_l^2(\bt)\Big)
		   \|f-\proj\|_2^2.
\end{align*}
Then, the Pythagoras theorem yields
\begin{align*}
E_{f}\|f-\proj\|_2^2&= \|f-\pf\|_2^2+E_{f}\|\pf	-\proj\|_2^2=\|\Pi_{S_F}f\|_2^2+E_{f}\|\pf	-\proj\|_2^2.
\end{align*}
This completes the proof (\ref{U1}). As for the variance of $U_{n,2}$, we have
$$
Var_f[U_{n,2}]=E_{f}E_f^{\mcD_2}[U_{n,2}^2]-(E_f[ U_{n,2}])^2=A_{n,1}+A_{n,2}+A_{n,3},
$$
where
\begin{align*}
A_{n,1}&=E_{f}\iint  \big(f-{\proj}\big)^2(\bt_1)\big(f-{\proj}\big)^2(\bt_2)G_n^2(\bt_1,\bt_2)d\bt_1d\bt_2,\\
A_{n,2}&=\frac{4}{m(m-1)}\binom{m}{3}E_{f}\iiint \big(f-{\proj}\big)^2(\bt_1)\big(f-{\proj}\big)(\bt_2)G_n(\bt_1,\bt_2)\\
       &\qquad\qquad\qquad\qquad\qquad\qquad\qquad\qquad \times\big(f-\proj\big)(\bt_3)G_n(\bt_1,\bt_3)d\bt_1d\bt_2d\bt_3,
\end{align*}
and
\begin{align*}
A_{n,3}= \frac{4}{m(m-1)}\binom{m}{4}E_{f}\Big\{\iint f(\bt_1)f(\bt_2)G_n(\bt_1,\bt_2)d\bt_1d\bt_2\Big\}^2 -(E_fU_{n,2})^2.
\end{align*}
Let us bound the first term $A_{n,1}$:
\begin{align*}
A_{n,1}&=E_{f}\sum_{l, {l'}\in\mcL(\bw_n)}w_{l,n}w_{l',n}\Big(\int (f-\proj)^2(\bt)\varphi_l(\bt)
\varphi_{{l'}}(\bt)\,d\bt\Big)^2.
\end{align*}
Now, in view of Bessel's inequality,
\begin{align*}
A_{n,1}
	&\le \max_{l\in\mcL(\bw_n)} w_{l,n}^2 E_{f,B}\sum_{l\in\mcL(\bw_n)}\int (f-\proj)^4(\bt)\varphi_l^2(\bt)\,d\bt\\
	&\le \Big(\max_{l\in\mcL(\bw_n)} w_{l,n}^2\Big)\Big(\sup_{\bt\in\Delta}\sum_{l\in\mcL(\bw_n)}\nolimits \varphi_l^2(\bt)\Big)
	E_{f}\big[\big\|f-\proj\big\|_4^4\big],
\end{align*}
and the expression inside the last expectation can be bounded using the inequality
$\big\|f-\proj\big\|_4^4\le 8(\|\Pi_{S_F}f\|_4^4+\|\pf-\proj\|_4^4)$.

The term $A_{n,2}$ can be dealt with similarly. Using the Cauchy-Schwarz inequality,
\begin{align*}
A_{n,2}
	&= \frac{4}{m(m-1)}\binom{m}{3} \sum_{l,l'\in \mcL(\bw_n)}w_{l,n}w_{l',n}\theta_l[f]\theta_{l'}[f]
	   E_{f}\Big\{\int \big(f-\proj\big)^2(\bt) \varphi_l(\bt)\varphi_{l'}(\bt)\,d\bt\Big\}\\
	&\le \binom{m}{2}^{1/2}\Big( \sum_{l}w_{l,n}^2\theta_l[f]^2\Big)
	    \Big( \sum_{l,l'\in \mcL(\bw_n)}\Big\{\int E_{f}\big[\big(f-\proj\big)^2(\bt)\big] \varphi_l(\bt)\varphi_{l'}(\bt)\,d\bt
		\Big\}^2\Big)^{1/2}\\
	&\leq \Big(\max_{l\in\mcL(\bw_n)} w_{l,n}\Big) h_n[f,\bw_n] \Big( \sum_{l,l'\in \mcL(\bw_n)}E_{f}\Big\{\int \big(f-\proj\big)^2(\bt)
        \varphi_l(\bt)\varphi_{l'}(\bt)\,d\bt\Big\}^2\Big)^{1/2}.
\end{align*}
By virtue of the Bessel inequality, it holds that
\begin{align*}
A_{n,2}
	&\leq \Big(\max_{l\in\mcL(\bw_n)} w_{l,n}\Big) h_n[f,\bw_n] \Big( \sum_{l\in \mcL(\bw_n)}\int E_{f}\big[\big(f-\proj\big)^4(\bt)\big]
	 \varphi_l^2(\bt)\,d\bt\Big)^{1/2}\\
	&\leq \Big(\max_{l\in\mcL(\bw_n)} w_{l,n}\Big) h_n[f,\bw_n] \Big(\sup_{\bt\in\Delta} \sum_{l\in \mcL(\bw_n)}\varphi_l^2(\bt)\,d\bt\Big)^{1/2}
	  \big(E_{f}[\|f-\proj\|^4_4]\big)^{1/2}.
\end{align*}
The last expectation can be bounded in the same way as we did several lines above for the term $A_{n,1}$.
The last term $A_{n,3}$ is actually negative
$$
A_{n,3}=
\frac{4}{m(m-1)}\binom{m}{4}\Big(\sum_{l\in\mcL(\bw_n)}\nolimits w_{l,n}\theta_l^2\Big)^2-
	\frac{m(m-1)}{2}\Big(\sum_{l\in\mcL(\bw_n)}\nolimits w_{l,n}\theta_l^2\Big)^2\leq 0.$$
Combining all these estimates, we get (\ref{U2}).
\end{proof}

\begin{lemma}\label{lem:CLT}
Let  $\bw_n=(w_{l,n})_{l\in\mcL}$ be a family of positive numbers containing only a finite number of nonzero entries and such that
$\sum_{l\in\mcL} w_{l,n}^2=1$. Assume that the random variable $\xi_1$ has finite fourth moment: $E_f[\xi_1^4]<\infty$. If, as $n\to\infty$,
\begin{align}
\|\bw_{n}\|_\infty =o(1)\qquad\text{and}\qquad \|\bw_n\|_\infty^2\Big(\sup_{\bt\in\Delta}\sum_{l\in\mcL(\bw_n)}\nolimits\varphi_l(\bt)^2\Big)^2=o(n),
\end{align}
then $U_{n,0}$ is asymptotically Gaussian $\mcN(0,1)$.
\end{lemma}
\begin{proof}
This result is an immediate consequence of \cite[Theorem 1]{Hall84}.
\end{proof}

With these tools at hand, we are now in a position to establish the asymptotic normality of the U-statistic $U_n$ which leads to
an evaluation of the type I error of the U-test. Let us recall that, for $f\in \mcF_0$, it holds $Q[f]=\sum q_l\theta_l[f]^2=0$
and, therefore, $\theta_l[f]=0$ for all $l\in{S_F}=\{l : q_l\neq 0\}$. Hence, for every $f\in\mcF_0$, $h_n[f,\bw_n]=0$ and $\Pi_{S_F}f=0$.
So, it follows from Lemma~\ref{lem:AD1} that under the assumptions of the proposition, the convergences $E_f[U_{n,1}^2]\to 0$ and
$E_f[U_{n,2}^2]\to 0$ hold true uniformly in $f\in\mcF_0$. This implies that $U_{n,1}$ and $U_{n,2}$ tend to zero in $P_f$-probability,
uniformly in $f\in\mcF_0$. On the other hand, according to Lemma~\ref{lem:CLT}, $U_{n,0}\to \mcN(0,1)$ in distribution. The claim of
the proposition follows from Slutsky's lemma.

\subsection{Proof of Proposition~\ref{prop_2}}

We first note that for every $\bar h>0$ it holds
\begin{align}\label{eq:35}
\sup_{f\in\mcF_1(\rho)} P_f(U_n\le u) =\Big(\sup_{\substack{f\in\mcF_1(\rho)\\h_n[f,\bw_n]> \bar h}} P_f(U_n\le u)\Big)
\bigvee \Big(\sup_{\substack{f\in\mcF_1(\rho)\\h_n[f,\bw_n]\le \bar h}} P_f(U_n\le u)\Big).
\end{align}
The value of  $\bar h$ will be made precise later in the proof. Assume merely  by now that $\bar h>2(1+u)$. Then,
\begin{align*}
\sup_{\substack{f\in\mcF_1(\rho);\\h_n[f,\bw_n]> \bar h}} P_{f}(U_n\leq {u})
&\leq \sup_{f\in\Sigma;h_n[f,\bw_n]> \bar h}\frac{\text{Var}_{f}[U_n]}{\big( E_{f}[U_n]-{u}\big)^2}
= \sup_{f\in\Sigma;h_n[f,\bw_n]> \bar h}\frac{\text{Var}_{f}[U_n]}{\big( h_n[f,\bw_n]-{u}\big)^2}.
\end{align*}
Using the conditions of the proposition and the inequalities of Lemma~\ref{lem:AD1}, we get that for some
constants $C,C'$ independent of $\bar h$,
\begin{align}\label{eq:36}
\sup_{\substack{f\in\mcF_1(\rho)\\h_n[f,\bw_n]> \bar h}} P_{f}(U_n\leq {u})
&\le \sup_{f\in\Sigma;h_n[f,\bw_n]> \bar h}\frac{C(1+h_n[f,\bw_n])}{\big( h_n[f,\bw_n]-{u}\big)^2}
\le C\frac{1+\bar h}{\big(\bar h-{u}\big)^2}\le C'\bar h^{-1}.
\end{align}
Let us switch to the second sup in (\ref{eq:35}). Let $\delta_n>0$ be a sequence tending to zero. One readily checks that
\begin{align}\label{eq:37}
P_f(U_n\le u)
	 &= P_f(h_n[f,\bw_n]+U_{n,0}+U_{n,1}+(U_{n,2}-h_n[f,\bw_n])\le u)\nonumber\\
	 &\le P_f(h_n[f,\bw_n]+U_{n,0}\le u+\delta_n)+P_f(-U_{n,1}-(U_{n,2}-h_n[f,\bw_n])\ge \delta_n)\nonumber\\
	 &\le F_{U_{0,n}}(u-h_n[f,\bw_n]+\delta_n)+\frac{2Var_f(U_{n,1})+2Var_f(U_{n,2})}{\delta_n^2},	
\end{align}
where $F_{U_{0,n}}(\cdot)$ is the c.d.f.\ of $U_{0,n}$.
On the one hand, we know from Lemma~\ref{lem:CLT} that $U_{n,0}$ converges in distribution to $\mcN(0,1)$.
This entails that $F_{U_{0,n}}$ converges uniformly over $\RR$ to  $\Phi$. Therefore,
$$
F_{U_{0,n}}(u-h_n[f,\bw_n]+\delta_n)=\Phi(u-h_n[f,\bw_n]+\delta_n)+o(1)=\Phi(u-h_n[f,\bw_n])+o(1)+\delta_n O(1).
$$
On the other hand, in view of Lemma~\ref{lem:AD1}, $Var_f(U_{n,1})+Var_f(U_{n,2})=O(\|\Pi_{S_F}f\|_4^4+\|\Pi_{S_F}f\|_2^2)$.

Then we have,
\begin{align*}
\|\Pi_{S_F}f\|_2^2
	&= \sum_{l\in S_F}\theta_l^2\leq \frac{1}{\zeta_n}\sum_{w_{l,n}\ge \zeta_n} w_{l,n}\theta_l ^2+\sum_{w_{l,n}< \zeta_n} \theta_l ^2\\
	& \leq \frac{\sqrt{2}\,h_n[f,\bw_n]}{\zeta_n(m-1)}+\sup_{l\in S_F :w_{l,n}<\zeta_n} c_l^{-1}.
\end{align*}
Applying H\"older's inequality we get
$\| \Pi_{S_F}f\|_4^4 \leq \|\Pi_{S_F}f\|_2^{2(p-4)/(p-2)}\|\Pi_{S_F}f\|_p^{2p/(p-2)}$.
Therefore, we have
\begin{align*}
\sup_{\substack{f\in\mcF_1(\rho)\\h_n[f,\bw_n]\le \bar h}} P_f(U_n\le u)
	&\le \sup_{f\in\mcF_1(\rho)} \Phi(u-h_n[f,\bw_n])+o(1)+\delta_nO(1)+\frac{o(1)(\bar h^{(p-4)/(p-2)}+\bar h)}{\delta_n^2}.
\end{align*}
Choosing $\bar h$ large enough and then making $\delta_n$ tend to zero sufficiently slowly we get the desired result.

\subsection{Proof of Proposition~\ref{prop_3}}

Using Kneser's minimax theorem for bilinear forms \citep{Kneser52}, we can interchange the sup and the inf as follows:
\begin{align}\label{eq:23}
\sup_{\substack{\bw\in\RR^{\mcL}_+\\ \|\bw\|_2=1}}
\inf_{\substack{\bv\in\RR^{\mcL}_+\\ \langle \bv,\bc\rangle \le 1, \langle \bv,\bq\rangle \ge \rho^2 }}
\langle \bw ,\bv\rangle
=
    \inf_{\substack{\bv\in\RR^{\mcL}_+\\ \langle \bv,\bc\rangle \le 1, \langle \bv,\bq\rangle \ge \rho^2 }}
    \sup_{\substack{\bw\in\RR^{\mcL}_+\\ \|\bw\|_2=1}}
    \langle \bw ,\bv\rangle
=   \inf_{\substack{\bv\in\RR^{\mcL}_+\\ \langle \bv,\bc\rangle \le 1, \langle \bv,\bq\rangle \ge \rho^2 }}
    \|\bv\|_2,
\end{align}
Furthermore, the array $\bw^*$ attaining the sup is given by $w^*_l=v_l/\|\bv\|_2$. Now, the minimization at the right-hand side
of (\ref{eq:23}) involves a convex second-order cost function $\|\bv\|_2^2$ and linear constraints $v_l\ge 0$, $\langle \bv,\bc\rangle \le 1$ and
$\langle\bv,\bq\rangle\le\rho^2$. Therefore, according to KKT conditions, if there exist $\mu,\lambda \ge 0$ and $\nnu\in\RR_+^\mcL$ satisfying for some
$\bv^*\in\RR^\mcL_+$ the conditions $2\bv^*+\lambda\bc-\mu\bq-\nnu=0$ and $\lambda(\langle\bv^*,\bc\rangle-1)=0$, $\mu(\langle\bv^*,\bq\rangle-\rho^2)=0$
and $\nu_l v_l^*=0$ for all $l$, then $\bv^*$ is a solution to the minimization problem (\ref{eq:23}). Under the conditions of the proposition, one easily
checks that these KKT conditions are fulfilled with $\lambda=2/\sum_l c_l(T_\rho q_l-c_l)_+$, $\mu=2T_\rho/\sum_l c_l(T_\rho q_l-c_l)_+$ and
$\nu_l=2(c_l-T_\rho q_l)_+/\sum_l c_l(T_\rho q_l-c_l)_+$.

\subsection{Proof of Theorem~\ref{thm_1}}
To ease notation, we set $\mcN_{n,\gamma}=\mcN(T_{n,\gamma})$.
We first check that under the assumptions of the theorem all the conditions required in Propositions \ref{prop_1} and \ref{prop_2}
are fulfilled. Since $\|\widehat\bw^*_n\|_0=|\mcN_{n,\gamma}|$ and $\|\widehat\bw^*_n\|_\infty^2 \le
\max_{l\in\mcN_{n,\gamma}}q_l^2/ \sum_{l\in\mcN_{n,\gamma}} \big(q_l-\frac{c_l}{T_{n,\gamma}}\big)^2$, condition \textbf{[C1]} implies
the first condition of Proposition~\ref{prop_1}. Conditions \textbf{[C3]} and \textbf{[C4]} imply respectively the third and the second conditions of Proposition~\ref{prop_1}.
Finally, condition \textbf{[C6]} implies the fourth condition of Proposition~\ref{prop_1}. Thus, we have checked that under the conditions of the theorem,
the claim of Proposition~\ref{prop_1} holds true. To check that the claim of Proposition~\ref{prop_2} holds true as well, it suffices to check the first
assumption of that proposition (the second one being identical to \textbf{[C7]}). In fact, it is not difficult to check that the first assumption of
Proposition~\ref{prop_2} follows from \textbf{[C2]}, \textbf{[C4]} and \textbf{[C5]} for the sequence $\zeta_n^2=\min_{l\in\mcN_{n,\gamma}}q_l^2/4\sum_{l\in\mcN_{n,\gamma}}q_l^2$.

Therefore, combining the results of Proposition~\ref{prop_1} and \ref{prop_2}, we get that
\begin{align}\label{eq:27}
\gamma_n(\mcF_0,\mcF_1(r_{n,\gamma}^*),\widehat\phi_n^*)\le \Phi(-z_{1-\gamma/2})+ \Phi\big(z_{1-\gamma/2}-\inf_{f\in\mcF_1(r_{n,\gamma}^*)}\nolimits h_n[f,\widehat\bw_n^*]\big)+o(1).
\end{align}
In view of Proposition~\ref{prop_3}, the infimum over $f$ of $h_n[f,\widehat\bw_n^*]$ can be evaluated as follows:
\begin{align*}
\inf_{f\in\mcF_1(r_{n,\gamma}^*)}\nolimits h_n[f,\widehat\bw_n^*]
	& = \Big(\frac{m(m-1)}{2}\Big)^{1/2}\inf_{\substack{\ttheta\in\RR^\mcL:\sum_l c_l\theta_l^2\le 1\\ \sum_l q_l\theta_l^2\ge (r_{n,\gamma}^*)^2}}
			\sum_l \widehat w_{l,n}^*\theta_l^2\\
	& = \Big(\frac{m(m-1)}{2}\Big)^{1/2}\inf_{\substack{\bv\in\RR^\mcL_+:\langle \bv,\bc\rangle\le 1\\ \langle \bv,\bq\rangle \ge (r_{n,\gamma}^*)^2}}
			\langle \widehat\bw_{n}^*,\bv\rangle\\
	& = \Big(\frac{m(m-1)}{2}\Big)^{1/2}\|\bv^*\|_2\\
	& = \Big(\frac{m(m-1)}{2}\Big)^{1/2}\frac{\big(\sum_{l\in\mcN_{n,\gamma}} (T_{n,\gamma}q_l-c_l)^2\big)^{1/2}}{\sum_{l\in\mcN_{n,\gamma}}
			 c_l(T_{n,\gamma}q_l-c_l)}.
\end{align*}
Inserting this expression in (\ref{eq:27}) and using (\ref{eq:24}), we get that
\begin{align*}
\gamma_n(\mcF_0,\mcF_1(r_{n,\gamma}^*),\widehat\phi_n^*)
		&\le \Phi(-z_{1-\gamma/2})+ \Phi\big(z_{1-\gamma/2}-2z_{1-\gamma/2}+o(1)\big)+o(1)\\
		&= 2\Phi(-z_{1-\gamma/2})+o(1)=\gamma+o(1).
\end{align*}

\subsection{ Proof of Theorem  \ref{thm_2}}\label{proof_thm_2}

The proof of the lower bound follows the steps of~\citep{ingster2009minimax}. However, we considerably modified
the way some of these steps are carried out which allowed us to relax several assumptions and resulted in a shorter
proof.

Let us recall that $\ttheta[f] =(\theta_l[f])_{l\in\mcL}\in\ell_2(\mcL)$ is the array of Fourier coefficients of a function in $L_2(\Delta)$
w.r.t.\ the system $(\varphi_l)_{l\in\mathcal{L}}$. We introduce the sets $\Theta_1(\rho)=\big\{\ttheta\in \ell_2(\mcL) : \langle \bc,\ttheta^2\rangle\leq 1,\
\langle \bq,\ttheta^2\rangle\geq\rho^2 \big\}$ and $\Theta_0=\big\{\ttheta\in\ell(\mcL) : \langle \bc,\ttheta^2\rangle\leq 1,\
\langle \bq,\ttheta^2\rangle=0\big\}$, where we used the notation $\ttheta^2=\{\theta_l^2\}_{l\in\mcL}$.
Clearly, if $f$ belongs to the functional class $\mcF_1(\rho)$ (resp.\ $\mcF_0$) then $\ttheta[f]\in\Theta_1(\rho)$
(resp.\ $\ttheta[f]\in\Theta_0$).

Let $C<1$ be a constant. Our goal is to prove that $\gamma_n(\mcF_0,\mcF_1(Cr_{n,\gamma}^*))\ge \gamma+o(1)$.
To get this lower bound, we define prior measures that are essentially concentrated on the sets $\Theta_0$ and $\Theta_1$.
Let $\pi_n^1$ and $\pi_n^2$ be measures on the space $\ell_2(\mcL)$ such that $\pi_n^1(\Theta_0)=1+o(1)$ and $\pi_n^2(\Theta_1(Cr_{n,\gamma}^*))=1+o(1)$.
Those priors lead to the corresponding mixtures:
$$
P_{\pi_n^i}(A)=\int P_{\ttheta}(A)\pi_n^i(d\ttheta) \quad \text{ for every measurable set } A\subset (\Delta\times\RR)^n, \quad i=1,2.
$$
If $\gamma_n(P_{\pi_n^1},P_{\pi_n^2})=\inf_{\psi:(\Delta\times\RR)^n\to\{1,2\}} \big\{P_{\pi_n^1}(\psi=2)+P_{\pi_n^2}(\psi=1)\big\}$ is the minimal total error probability for testing the simple
null hypothesis $H_0 : P=P_{\pi_n^1}$ against the simple alternative $H_1 : P=P_{\pi_n^2}$, then we
have (see Proposition 2.11 in \cite{ingster2003nonparametric})
$$
\gamma_n\big(\mcF_0, \mcF_1(Cr_{n,\gamma}^*)\big)\geq \gamma_n(P_{\pi_n^1},P_{\pi_n^2})+o(1).
$$
As shows the next result, to get the desired lower bound,  it suffices to show that the Bayesian log-likelihood
$\log(dP_{\pi_n^2}/dP_{\pi_n^1})$ is asymptotically equivalent to a Gaussian log-likelihood.

\begin{lemma}[section 4.3.1 in  \cite{ingster2003nonparametric}]\label{lem:1}
If there exists a deterministic sequence $u_n$ and a sequence of random  variables $\eta_n$ such that under $P_{\pi_n^1}$-probability
$\eta_n$ converges in distribution to $\mcN(0,1)$ and
\begin{equation}\label{L_n}
\log (dP_{\pi_n^2}/dP_{\pi_n^1})=u_n\eta_n-\frac{u_n^2}{2}+o_P(1),
\end{equation}
then
$\gamma_n(P_{\pi_n^1},P_{\pi_n^2})\geq 2\Phi(-u_n/2)+o(1)$.
\end{lemma}

For our purposes, we choose $\pi_n^1$ to be the Dirac measure in $\boldsymbol{0}$ and denote the
corresponding mixture probability $P_{\pi_n^1}$ by $P_0$. It is clear that with this choice $\pi_n^1(\Theta_0)=1$.
We now explain how $\pi_n^2$, that we will call $\pi_n$ from now on, is built. Let $\ba_n\in\RR_+^{\mcL}$ be an
array containing a finite number of nonzero elements. Let $\mcL(\ba_n)$ be the support of $\ba_n$, \textit{i.e.},
$a_l\not=0$ if and only if $l\in\mcL(\ba_n)$. We assume that $\mcL(\ba_n)\subset S_F$ and define $\pi_n(d\ttheta)$
as the Gaussian product measure such that under $\pi_n$ the entries $\theta_{l}$ are independent Gaussian with zero
mean and variance $a_l$.

\begin{proposition}\label{prop_4}
Let $\delta\in(0,1)$ be such that $1-\delta\ge C$.
Assume that $\ba_n=(1-\delta)\bv_n$ and, as $n\to\infty$, the following assumptions are fulfilled:
\begin{itemize}
\item[] {\bf [L1]} $\langle \bc,\bv_n\rangle\le 1$ and $\langle \bq,\bv_n\rangle\ge (r_{n,\gamma}^*)^2$,
\item[] {\bf [L2]} ${\max_{l\in\mcL(\bv_n)} (q_lv_l)}=o(\langle \bq,\bv\rangle)$ and ${\max_{l\in\mcL(\bv_n)} (c_lv_l)}=o(\langle \bc,\bv\rangle)$,
\item[] {\bf [L3]} $\|\bv_n\|_0\to \infty$ and  $n\|\bv_n\|_\infty^2 \|\bv_n\|_0^{2}\log \|\bv_n\|_0\to 0$,
\item[] {\bf [L4]} $n\|\bv_n\|_\infty \|\bv_n\|_0^{1/3}\to 0$ and $\|\bv_n\|_3=o(\|\bv_n\|_2)$.
\item[] {\bf [L5]} For some $L_5>0$, it holds $\sum_{l\in\mcL(\ba_n)}\varphi_l^2(\bt)\le L_5\|\ba_n\|_0$.
\end{itemize}
Then, as $n\to\infty$,
\begin{equation}\label{eq:20}
\gamma_n(\mcF_0,\mcF_1(Cr_{n,\gamma}^*))\geq 2\Phi\Big(-\frac{n(1-\delta)}{2\sqrt{2}}\|\bv_n\|_2\Big)+o(1).
\end{equation}
\end{proposition}
\begin{proof}
The proof of this proposition will be carried out with the help of several lemmas.
The fact that $\pi_n\big(\Theta_1(Cr_{n,\gamma}^*)\big)=1+o(1)$ is proved in the following lemma.
\begin{lemma}\label{lem:2}
Assume that $\ba_n=(1-\delta)\bv_n$ satisfies {\bf [L1]} and {\bf [L2]}.
Then, for every $\delta \in (0,1)$, it holds that $\pi_n\big(\Theta_1(Cr_n^*)\big)=1+o(1)$.
\end{lemma}
\begin{proof}
Let us denote  $\mathcal{H}_1(\ttheta)=\sum_{l\in\mathcal{L}} q_l \theta_l^2$ and $\mathcal{H}_2(\ttheta)=\sum_{l\in\mathcal{L}}  c_l\theta_l^2$.
In view of {\bf [L1]}, we have
$$
\int \mathcal{H}_1(\ttheta)\pi_n(d\ttheta)=\sum_{l\in\mathcal{L}} q_l a_l\ge (r_{n,\gamma}^*)^2(1-\delta),\qquad
\int \mathcal{H}_2(\ttheta)\pi_n(d\ttheta)=\sum_{l\in\mathcal{L}} c_l a_l\le 1-\delta.
$$
On the other hand, since the variance of the sum of independent random variables equals the sum
of the variances of these random variables, we get
\begin{align*}
\int\mathcal{H}_1(\ttheta)^2\pi_n(d\ttheta)-\Big(\int\mathcal{H}_1(\ttheta)\pi_n(d\ttheta)\Big)^2
		& = 2\sum_{l\in\mathcal{L}} q_l^2 a_l^2\le 2\langle \bq, \ba_n\rangle\max_{l\in\mcL(\ba_n)} (q_l a_l).
\end{align*}
By Tchebychev's inequality, we arrive at
\begin{align*}
\pi_n\big(\ttheta:\mcH_1(\ttheta)< (C r_{n,\gamma}^*)^2\big)
        &\le \frac{2\max_{l\in\mcL(\bv_n)} (q_l v_l)}{C^2(1-C)^2 \langle \bq, \bv_n\rangle},\\
\pi_n\big(\ttheta:\mcH_2(\ttheta)>1\big)
        &\le \frac{2\max_{l\in\mcL(\bv_n)} (c_l v_l)}{\delta^2\langle \bc, \bv_n\rangle}.
\end{align*}
The claim of the lemma follows now from condition {\bf [L2]}.
\end{proof}

Second, we show that for every $p>2$ and every $L>0$, the probability $\pi_n(\ttheta: \|\sum_l \theta_l\varphi_l\|_p>L)$ tends to zero.
Indeed, in view of the Tchebychev inequality and Fubini's theorem,
\begin{align*}
\pi_n\Big(\ttheta: \Big\|\sum_l\nolimits \theta_l\varphi_l\Big\|_p>L\Big) & \le L^{-p} \int_\Delta E_{\pi_n}\bigg[\bigg|\sum_{l} \theta_l\varphi_l(\bt)\bigg|^p\bigg]\,d\bt.
\end{align*}
Using the fact that for every fixed $\bt$, the random variable $\sum_{l} \theta_l\varphi_l(\bt)$ is Gaussian with zero mean and
variance $\sum_l a_l \varphi_l^2(\bt)$, we get
\begin{align*}
\pi_n\Big(\ttheta: \Big\|\sum_l\nolimits \theta_l\varphi_l\Big\|_p>L\Big) & \le p! L^{-p} \int_\Delta \bigg|\sum_{l} a_l\varphi_l^2(\bt)\bigg|^{p/2}\,d\bt
\le p! L_5^{p/2} L^{-p} (\|\ba_n\|_\infty\|\ba_n\|_0)^{p/2}.
\end{align*}
The last expression tends to zero as $n\to\infty$ in view of condition \textbf{[L3]}.

We focus now on the proof of (\ref{L_n}). Set $m = |\mcL(\ba_n)|$ and let $\Phi_n$ be the $m\times n$ matrix having as generic element
$(\Phi_n)_{li}=\varphi_l(\bt_i)$.  Let $\ttA_n$ be $m\times m$ diagonal matrix having the nonzero entries
of $\ba_n$ on its main diagonal. It is clear that under $P_{\pi_n}$, conditionally to $\mcT_n$, $\bx=(x_1,\ldots,x_n)^\top$
is distributed according to a multivariate Gaussian distribution with zero mean and $n\times n$ covariance matrix
$\ttR_n= \Phi_n^\top\ttA_n\Phi_n+\ttI_n$. Therefore, the logarithm of its density w.r.t.\ $P_0$ is given by
$$
\log\Big(\frac{dP_{\pi_n}}{dP_0}(\bx;\bt_1,\ldots,\bt_n)\Big)=-\frac12\big(\log\det\ttR_n+\bx^\top(\ttR_n^{-1}-\ttI_n)\bx\big).
$$
In what follows, we denote by $\matnorm{\ttM}=\sup_{\|\bx\|_2=1} \|\ttM\bx\|_2$ the spectral norm of a matrix $\ttM$.
\begin{lemma}
Let $\bar\ttR_n=n\ttA_n+\ttI_m$ and $m=m_n\to\infty$. If $n^2\|\ba_n\|^2_\infty\|\ba_n\|_0\matnorm{\frac1n\Phi_n\Phi_n^\top-\ttI_m}^2=o_P(1)$
and $|\Tr[\bar\ttR_n^{-1}\ttB_n]|+E[|\xxi^\top\bar\ttR_n^{-1}\ttB_n\bar\ttR_n^{-1}\xxi|]=o_P(1)$, then under $P_0$ it holds $\log\big({dP_{\pi_n}}/{dP_0}\big)
= -\frac12\big(\log\det\bar\ttR_n+\xxi^\top(\bar\ttR_n^{-1}-\ttI_m)\xxi\big)+o_P(1)$, where $\xxi\sim\mcN_m(0,\ttI_m)$.
\end{lemma}

\begin{proof}
Let us denote $\tilde\ttR_n=\ttA_n^{1/2}\Phi_n\Phi_n^\top\ttA_n^{1/2}+\ttI_m$, $\ttB_n=\tilde\ttR_n-\bar\ttR_n$ and introduce
the function $g(z)=\log\det(\bar\ttR_n+z\ttB_n)$ for $z\in[0,1]$.  One easily checks that
$g(1)=\log\det\tilde\ttR_n=\log\det\ttR_n$, $g(0)=\log\det\bar\ttR_n$ and $g'(z)=\Tr[(\bar\ttR_n+z\ttB_n)^{-1}\ttB_n]$.
Therefore, the relation $g(1)-g(0)=g'(\bar z)$ for some $\bar z\in[0,1]$ implies
\begin{align*}
|\log\det\ttR_n-\log\det\bar\ttR_n |
		&=|\Tr[(\bar\ttR_n+\bar z\ttB_n)^{-1}\ttB_n]|\\
		&\le |\Tr[\bar\ttR_n^{-1}\ttB_n]|+ m \matnorm{(\bar\ttR_n+\bar z\ttB_n)^{-1}-\bar\ttR_n^{-1}}\matnorm{\ttB_n}.
\end{align*}
Using the identity $(\bar\ttR_n+\bar z\ttB_n)^{-1}-\bar\ttR_n^{-1}= -\bar z (\bar\ttR_n+\bar z\ttB_n)^{-1}\ttB_n\bar\ttR_n^{-1}$, we get
\begin{align*}
|\log\det\ttR_n-\log\det\bar\ttR_n |
		&\le |\Tr[\bar\ttR_n^{-1}\ttB_n]|+ m \matnorm{(\bar\ttR_n+\bar z\ttB_n)^{-1}}\matnorm{\bar\ttR_n^{-1}}\matnorm{\ttB_n}^2\\
        &\le |\Tr[\bar\ttR_n^{-1}\ttB_n]|+ m \matnorm{\ttB_n}^2,
\end{align*}
where we used that $\bar\ttR_n$ and $\bar\ttR_n+\bar z\ttB_n=\ttI_m+\bar z\ttA_n^{1/2}\Phi_n\Phi_n^\top\ttA_n^{1/2}+(1-\bar z)n\ttA_n$ have all
their eigenvalues $\ge 1$.
On the other hand, one can check that $\matnorm{\ttB_n}\le n\matnorm{\ttA_n}\matnorm{\frac1n\Phi\Phi^\top-\ttI_m}$. Combining these
inequalities with the facts $\matnorm{\ttA_n}=\|\ba_n\|_\infty$ and $m=\|\ba_n\|_0\to \infty$ we arrive at
$\log\det\ttR_n=\log\det\bar\ttR_n +o_P(1)$.

The term $\bx^\top\ttR_n^{-1}\bx$ is dealt with similarly. First, using the singular values decomposition of the matrix $\ttA_n^{1/2}\Phi_n$,
one can note that for an appropriately chosen vector $\xxi\sim\mcN_m(0,\ttI_m)$, it holds that
$\bx^\top(\ttR_n^{-1}-\ttI_n)\bx=\xxi^\top(\tilde\ttR_n^{-1}-\ttI_m)\xxi$. Then, we introduce the function
$\bar g(z)=\xxi^\top[\bar\ttR_n+z\ttB_n]^{-1}\xxi$, the derivative of which is given by $g'(z)=-\xxi^\top(\bar\ttR_n+z\ttB_n)^{-1}\ttB_n(\bar\ttR_n+z\ttB_n)^{-1}\xxi$. Therefore, for some $\bar z\in[0,1]$,
\begin{align*}
|\xxi^\top\tilde\ttR_n^{-1}\xxi-\xxi^\top\bar\ttR_n^{-1}\xxi|
		&=|\xxi^\top(\bar\ttR_n+\bar z\ttB_n)^{-1}\ttB_n(\bar\ttR_n+\bar z\ttB_n)^{-1}\xxi|\\
        &\le |\xxi^\top\bar\ttR_n^{-1}\ttB_n\bar\ttR_n^{-1}\xxi|+|\xxi^\top[(\bar\ttR_n+\bar z\ttB_n)^{-1}-\bar\ttR_n]^{-1}\ttB_n(\bar\ttR_n+\bar z\ttB_n)^{-1}\xxi|\\
        &\qquad+|\xxi^\top[(\bar\ttR_n+\bar z\ttB_n)^{-1}-\bar\ttR_n]^{-1}\ttB_n\bar\ttR_n^{-1}\xxi|\\
        &\le |\xxi^\top\bar\ttR_n^{-1}\ttB_n\bar\ttR_n^{-1}\xxi|+2\|\xxi\|_2^2\matnorm{\ttB_n}^2.
\end{align*}
It is well-known that $\|\xxi\|_2^2$ being distributed according to the $\chi^2_m$ distribution is $O_P(m)$, as $m\to\infty$. This completes the
proof of the lemma.
\end{proof}
According to \cite[Cor.~5.52]{Versh}, under {\bf [C3]}, we have $\matnorm{\frac1n\Phi_n\Phi_n^\top-\ttI_m}\le C(\frac{m\log m}{n})^{1/2}$
with probability at least $1-1/n$. Furthermore, using the facts that the $\bar\ttR_n$ is a diagonal matrix with diagonal entries $\ge 1$ and that the
variance of the sum of independent random variables equals the sum of variances,
one readily checks that $E|\Tr[\bar\ttR_n^{-1}\ttB_n]|^2+E[|\xxi^\top\bar\ttR_n^{-1}\ttB_n\bar\ttR_n^{-1}\xxi|^2]\le 3C_3^2 n \|\bv_n\|^2_\infty\|\bv_n\|^2_0$.
Hence, condition {\bf [L3]} implies that the two conditions of the last lemma are fulfilled and, therefore, its
claim holds true. Using the fact that $\ttA_n$ is diagonal, we get
\begin{align}
\log\big({dP_{\pi_n}}/{dP_0}\big)
		&= \frac12\sum_l\Big(\frac{na_l\xi_l^2}{na_l+1}-\log(na_l+1)\Big)+o_P(1)\nonumber\\
		&= \frac12\sum_l\Big(\frac{na_l}{na_l+1}-\log(na_l+1)\Big)+\sum_l\frac{na_l(\xi_l^2-1)}{2(na_l+1)}+o_P(1).\label{eq:18}
\end{align}
\begin{lemma}
Let us denote
$$
u_n=\frac{n\|\ba_n\|_2}{\sqrt2},\quad \eta_n=\frac1{u_n}\sum_{l\in\mcL}\frac{na_l(\xi_l^2-1)}{2(na_l+1)}.
$$
If the conditions $mn^3\|\ba_n\|_\infty^3\to 0$, and $\|\ba_n\|_3=o(\|\ba_n\|_2)$ are fulfilled, then $\eta_n$ converges in distribution to $\mcN(0,1)$ and
\begin{align}\label{eq:21}
\frac12\sum_{l\in\mcL}\Big(\frac{na_l}{na_l+1}-\log(na_l+1)\Big)+\sum_{l\in\mcL}\frac{na_l(\xi_l^2-1)}{2(na_l+1)}= u_n\eta_n-\frac{u_n^2}{2}+o(1).
\end{align}
\end{lemma}

\begin{proof}
Since $n\|\ba\|_\infty\to 0$, we have $\frac{na_l}{na_l+1}=na_l-(na_l)^2+O((na_l)^3)$ and $\log(na_l+1)=na_l-\frac{(na_l)^2}2+O((na_l)^3)$.
This implies that $\sum_{l\in\mcL}\big(\frac{na_l}{na_l+1}-\log(na_l+1)\big)=-\frac12 u_n^2+O(mn^3\|\ba_n\|_\infty^3)$. On the other hand, using
the central limit theorem for triangular arrays, we get the weak convergence of $\eta_n$ to $\mcN(0,1)$ provided that
$u_n^{-3}\sum_l (na_l)^3/(na_l+1)^3$ tends to zero. Since under the conditions of the lemma this convergence trivially holds, we get the
claim of the lemma.
\end{proof}
Combining Lemma~\ref{lem:1} with (\ref{eq:18}) and (\ref{eq:21}), we get (\ref{eq:20}) and the proposition follows.
\end{proof}

To complete the proof of Theorem~\ref{thm_2}, we shall show now that if we choose $T_{n,\gamma}$ as in Theorem~\ref{thm_1}
and define $\bv_n$ by
$$
v_l=v_{l,n}=\frac{(T_{n,\gamma}q_l-c_l)_+}{\sum_{l\in\mcL} c_l(T_{n,\gamma}q_l-c_l)_+},
$$
then all the conditions of Proposition~\ref{prop_4} are fulfilled. We start by noting that {\bf [L1]} is
straightforward. To check the first relation in {\bf [L2]}, we use {\bf [C1]} and $|\mcN(T_{n,\gamma})|\to\infty$, along
with the following evaluations:
\begin{align*}
\forall l\in\mcN(T_{n,\gamma}),\qquad\frac{q_lv_l}{\langle \bq,\bv\rangle}
         & = \frac{q_l(T_{n,\gamma}q_l-c_l)}{\sum_l q_l(T_{n,\gamma}q_l-c_l)_+}\le \frac{q_l^2}{\sum_l (q_l-\frac{c_l}{T_{n,\gamma}})_+^2}
          \le \frac{C_1}{|\mcN(T_{n,\gamma})|}.
\end{align*}
For the second relation in {\bf [L2]}, in view of (\ref{eq:24}), $\forall l\in\mcN(T_{n,\gamma})$ we have
\begin{align*}
\frac{c_lv_l}{\langle \bc,\bv\rangle}
         & = \frac{c_l(T_{n,\gamma}q_l-c_l)}{\sum_l c_l(T_{n,\gamma}q_l-c_l)_+}
         \le \frac{T_{n,\gamma}c_lq_l}{\sum_l c_l(T_{n,\gamma}q_l-c_l)_+}\\
         &\le \frac{T_{n,\gamma}c_lq_l\; O(1)}{nT_{n,\gamma}\big(\sum_l (q_l-\frac{c_l}{T_{n,\gamma}})_+^2\big)^{1/2}}
          \le \frac{\max_{l\in\mcN(T_{n,\gamma})}c_l}{n|\mcN(T_{n,\gamma})|^{1/2}}\; O(1).
\end{align*}
The last term tends to zero due to {\bf [C9]}.
From the definition of $\bv_n$, equation (\ref{eq:24}) and condition {\bf [C1]} one can deduce that
\begin{align*}
\|\bv_n\|_\infty
	&= \frac{\max_l (T_{n,\gamma}q_l-c_l)_+}{\sum_l c_l(T_{n,\gamma}q_l-c_l)_+}\le \frac{T_{n,\gamma}\max_l q_l}
			{n\big(\sum_l(T_{n,\gamma}q_l-c_l)_+^2\big)^{1/2}}O(1)\\
	&\le\frac{\max_l q_l}{n |\mcN(T_{n,\gamma})|^{1/2}\max_l q_l}O(1)= \frac{O(1)}{n |\mcN(T_{n,\gamma})|^{1/2}}. 		
\end{align*}
This inequality yields $n\|\bv_n\|_\infty^2\|\bv_n\|_0^2=O(|\mcN(T_{n,\gamma})|/n)$. Therefore, {\bf [L3]} follows from
{\bf [C8]}. Finally, to check that  {\bf [L4]} is true, we notice that	
$n\|\bv_n\|_\infty\|\bv_n\|_0^{1/3}=O(|\mcN(T_{n,\gamma})|^{\frac13-\frac12})=o(1)$
and
$$
\frac{\|\bv_n\|_3^3}{\|\bv_n\|_2^3} =\frac{\sum_l (T_{n,\gamma}q_l-c_l)_+^3}{\big(\sum_l (T_{n,\gamma}q_l-c_l)_+^2\big)^{3/2}}
\le \frac{\max_l q_l}{\big(\sum_l (q_l-\frac{c_l}{T_{n,\gamma}})_+^2\big)^{1/2}}\le \frac{C_1^{1/2}}{|\mcN(T_{n,\gamma})|^{1/2}}.
$$
Thus, all the conditions of Proposition~\ref{prop_4} are fulfilled and, therefore,
$$
\gamma_n(\mcF_0,\mcF_1(Cr_{n,\gamma}^*))\geq 2\Phi\Big(-\frac{n(1-\delta)}{2\sqrt{2}}\|\bv_n\|_2\Big)+o(1).
$$
Since this equation is true for every $\delta\in (0,1-C)$, it is also true for $\delta=0$, and
the claim of Theorem~\ref{thm_2} follows from (\ref{eq:24}).

\section{Proofs of lemmas and propositions of Section~\ref{sec:examples}}

\subsection{Proof of Lemma~\ref{lem:3}}

Let us write $\pf=\Pi_1f+\Pi_2f$, where $\Pi_1$ and $\Pi_2$ are the orthogonal projectors in $L_2(\Delta)$
onto the subspaces $\text{span}\{\varphi_l:l\in\mcN_1(T)\}$ and $\text{span}\{\varphi_l:l\in\mcN_2(T)\}$, respectively.
We first assume that the inequality $\sum_l c_l^{-1}<\infty$ is fulfilled.

On the one hand, using the Cauchy-Schwarz inequality,
\begin{align*}
\|\Pi_2f\|_4^4 & = \int_\Delta \Big(\sum_{l\in\mcN_2(T)}\nolimits\theta_l[f]\varphi_l(\bt)\Big)^4\,d\bt
				\le 2^{2d}\Big(\sum_{l\in\mcN_2(T)}\nolimits|\theta_l[f]|\Big)^4\\
		&\le 2^{2d}\Big(\sum_{l\in\mcN_2(T)}\nolimits c_l\theta_l[f]^2\Big)^2\Big(\sum_{l\in\mcN_2(T)}\nolimits c_l^{-1}\Big)^2
		\le 2^{2d}\Big(\sum_{l\in\mcN_2(T)}\nolimits c_l^{-1}\Big)^2.
\end{align*}
On the other hand,
\begin{align*}
\|\Pi_1f-\proj\|_4^4
		&= \int_\Delta\bigg(\sum_{l\in\mcN_1(T)} \big(\widehat\theta_l-\theta_l[f]\big)\varphi_l(\bt)\bigg)^4d\bt\\
		&= \int_\Delta\bigg(\frac1n\sum_{i=1}^n\sum_{l\in\mcN_1(T)} \big(x_i\varphi_l(\bt_i)-\theta_l[f]\big)\varphi_l(\bt)\bigg)^4d\bt.
\end{align*}
Using Fubini's theorem and Rosenthal's inequality, for some constant $C>0$, we get
\begin{align*}
E_f\|\Pi_1f-\proj\|_4^4
		&\le \frac{C}{n^4}\int_\Delta \sum_{i=1}^n E_f\bigg(\sum_{l\in\mcN_1(T)} \big(x_i\varphi_l(\bt_i)-\theta_l[f]\big)\varphi_l(\bt)\bigg)^4d\bt\\
		&\qquad+\frac{C}{n^4}\int_\Delta \bigg\{\sum_{i=1}^n E_f\bigg(\sum_{l\in\mcN_1(T)} \big(x_i\varphi_l(\bt_i)-\theta_l[f]\big)\varphi_l(\bt)\bigg)^2\bigg\}^2d\bt
\end{align*}
By H\"older's inequality, we get
\begin{align*}
E_f\bigg(\sum_{l\in\mcN_1(T)} \big(x_i\varphi_l(\bt_i)-\theta_l[f]\big)\varphi_l(\bt)\bigg)^4
	&\le |\mcN_1(T)|^3 \sum_{l\in\mcN_1(T)} E_f \big(x_i\varphi_l(\bt_i)-\theta_l[f]\big)^4\varphi_l(\bt)^4\\
	&\le 2^{2d}|\mcN_1(T)|^3 \sum_{l\in\mcN_1(T)} E_f \big(f(\bt_i)\varphi_l(\bt_i)+\xi_i\varphi_l(\bt_i)-\theta_l[f]\big)^4\\
	& = O(|\mcN_1(T)|^4),
\end{align*}
where we used the fact that $E[\xi^4]<\infty$ and that $E[f(t_i)^4]\le 2^{2d}(\sum_l c_l^{-1})^2<\infty$ under the conditions of the lemma.
Similar arguments lead to
$$
\int_\Delta \bigg\{\sum_{i=1}^n E_f\bigg(\sum_{l\in\mcN_1(T)} \big(x_i\varphi_l(\bt_i)-\theta_l[f]\big)\varphi_l(\bt)\bigg)^2\bigg\}^2d\bt
=O(n^2|\mcN_1(T)|^4),
$$
which implies that $E_f\|\Pi_1f-\proj\|_4^4=O(|\mcN_1(T)|^4/n^2)$. Combining the obtained evaluations, we get
$$
E_f\|\Pi f-\proj\|_4^4\le \frac{|\mcN_1(T)|^4}{n^2}+C\Big(\sum_{l:c_l>T}\nolimits c_l^{-1}\Big)^2.
$$
The required consistency follows from the assumption $|\mcN_1(T_n)|=o(n^{1/2})$.

Let us consider the case $\Sigma\subset W_2^\ssigma(R)$. Without loss of generality, we will assume that
$\Sigma= W_2^\ssigma(R)$ and $c_l=\sum_{i=1}^d (2\pi l_i)^{2\sigma_i}/R^2$. The computations remain the same
as in the previous case but the term $\|\Pi_2 f\|_4^4$ is bounded using Sobolev inequality \citep{Kolyada}.
Indeed, choosing $\ssigma'$ so that $\sigma'_i=(1-\tau)\sigma_i$ and $\tau< 1-d/(4\bar\sigma)$ (this implies that
$\bar\sigma'>d/4$), we get
\begin{align*}
\|\Pi_2f\|_4^2 & \le C \|\Pi_2 f\|_{W_2^{\ssigma'}}^2 = C \bigg[\sum_{\bl\in\mcN_2(T)} \sum_{i=1}^d (2\pi l_i)^{2\sigma_i'}\theta_\bl[f]^2 \bigg]
	\le C \bigg[\sum_{\bl\in\mcN_2(T)} d^\tau\big(c_lR^2\big)^{1-\tau}\theta_\bl[f]^2 \bigg]\\
	&\le C(d/T)^{\tau}R^{2(1-\tau)}\bigg[\sum_{\bl\in\mcN_2(T)} c_l\theta_\bl[f]^2 \bigg]
	\le C (d/T)^{\tau}R^{2(1-\tau)}.
\end{align*}
This completes the proof, since the last term tends to zero as $T\to\infty$.

\subsection{Proof of Lemma~\ref{lem:4}}

Let us introduce $\Pi_J f=\sum_{\bk\in [1,2^J]^d}\nolimits\alpha_{J,\bk}\varphi_{J,\bk}$.
We first decompose the empirical coefficients as follows:
$$
\widehat{\alpha}_{J,\bk}=\frac{1}{n}\sum_{i=1}^n \varphi_{J,\bk}(\bt_i)x_i=\frac{1}{n}\sum_{i=1}^n \varphi_{J,\bk}(\bt_i)f(\bt_i)+\frac{1}{n}\sum_{i=1}^n \varphi_{J,\bk}(\bt_i)\xi_i:=\tilde{\alpha}_{J,\bk}+\epsilon_{j,\bk}.
$$
Then, using standard arguments, we have
$$
\Big\| \pf-\proj\Big\|_4^4\leq 3^3\Big(\Big\|\sum_{\bk\in [1,2^J]^d}(\alpha_{J,\bk}-\tilde{\alpha}_{J,\bk})\varphi_{J,\bk}\Big\|_4^4+ \Big\|\sum_{\bk\in [1,2^J]^d}\epsilon_{J,\bk}\varphi_{J,\bk}\Big\|_4^4+\Big\|\pf-\Pi_Jf\Big\|_4^4\Big)
$$
with $\big\|\pf-\Pi_Jf\big\|_4^4=O(2^{-4J\sigma})$. Furthermore, by well-known properties of wavelet bases \citep{cohen2003numerical}
and the Rosenthal inequality,
$$
E_f\Big\|\sum_{\bk\in [1,2^J]^d}(\alpha_{J,\bk}-\tilde{\alpha}_{J,\bk})\varphi_{J,\bk}\Big\|_4^4=O(2^{Jd})\sum_{\bk} E_f(\alpha_{J,\bk}-\tilde{\alpha}_{J,\bk})^4=O\bigg(\frac{2^{2Jd}}{n^2}\bigg)
$$
and
\begin{align*}
E\Big\|\sum_{\bk\in [1,2^J]^d}\epsilon_{J,\bk}\varphi_{J,\bk}\Big\|_4^4
&=O(2^{Jd})\sum_{\bk} E[\epsilon_{J,\bk}^4]=O(2^{Jd})\sum_{\bk}E\Big(\frac{1}{n^2}\sum_{i=1}^n \varphi^2_{J,\bk}(\bt_i)\Big)^2\\
&=2^{2Jd} O\Big(\frac{1}{n^2}+\frac{2^{Jd}}{n^3}\Big).
\end{align*}
Finally we obtain, uniformly over $f\in\Sigma$,
$E_f\|\pf-\proj\|^4= O\big(\frac{2^{2Jd}}{n^2}+\frac{2^{3Jd}}{n^3}+2^{-4J\sigma}\big)$, and the announced result follows.

\section{Proof of Proposition~\ref{premierex}}

We are going to check that all the assumptions of Theorem~\ref{thm_1} and Theorem~\ref{thm_2} are satisfied.
We can use the Sobolev embedding theorem \citep{Kolyada} for \textbf{[C7]}: if $\bar\sigma>d/4$, then \textbf{[C7]} is satisfied.
For the pilot estimator proposed in subsection~\ref{ssec:3.15}, \textbf{[C6]}
holds as well. Since the Fourier basis is uniformly bounded, checking \textbf{[C3]}  is straightforward.

Let now $T_{n,\gamma}= (C_\gamma^*r_n^*)^{-2}(1+2\kappa^{-1})$, where $r_{n}^*$ and $C_\gamma^*$ are defined in Proposition~\ref{premierex}.
We will show that
\begin{itemize}
\item $T_{n,\gamma}$ satisfies (\ref{eq:24}),
\item $r_{n,\gamma}^*$ defined by (\ref{eq:25}) satisfies $r_{n,\gamma}^*\sim C_\gamma^*r_n^*$,
\item conditions \textbf{[C1]}, \textbf{[C2]}, \textbf{[C5]}, \textbf{[C8]} and \textbf{[C9]} are fulfilled.
\end{itemize}
To this end,  we need an asymptotic analysis of the terms
$$
I_0(T)= \sum_{\bl\in\ZZ^d} \Big(q_\bl-\frac{c_\bl}{T}\Big)_+^2,\qquad I_1(T) = \sum_{\bl\in\ZZ^d} q_\bl\Big(q_\bl-\frac{c_\bl}{T}\Big)_+
$$
and $I_2(T)=I_1(T)-I_0(T)$. For the first one, it holds that
\begin{align*}
I_0(T)&=\sum_{\bl\in\ZZ^d} \Big( \prod_{j=1}^{d}(2\pi l_j)^{2\alpha_j}- \sum_{i=1}^d \frac{(2\pi l_i)^{2\sigma_i}}{T}\Big)^2_+.
\end{align*}
For every $i\in\{1,\ldots, d\}$, we set
$$
m_i=\frac{T^{\gamma_i}}{2\pi},\quad \gamma_i=\frac{1}{2\sigma_i(1-\delta)}\ \text{ and } \ x_{\bl,i}=\frac{2\pi l_i}{T^{\gamma_i}}=\frac{l_i}{m_i}.
$$
 Note that, as $\delta<1$, we have $\gamma_i>0$. With this notation,
\begin{align*}
I_0(T)&=T^{\frac{2\delta}{1-\delta}}m_1\cdot\ldots\cdot m_d \sum_{\bl\in\ZZ^d} \Big(  \prod_{j=1}^{d}|x_{\bl,j}|^{2\alpha_j}
    -\sum_{i=1}^d |x_{\bl,i}|^{2\sigma_i}\Big)^2_+/(m_1\cdot\ldots\cdot m_d).
\end{align*}
As $m_i\to\infty$ for every $i$, we can replace the sums by integrals
$$
I_0(T)\sim \frac{T^{\frac{4\delta\bar\sigma+d}{2(1-\delta)\bar\sigma}}}{(2\pi)^d}
\int_{\sum_{i=1}^d |x_j|^{2\sigma_j}<\prod_{j=1}^{d}|x_j|^{2\alpha_j} }
\Big(\prod_{j=1}^{d}|x_j|^{2\alpha_j} -\sum_{i=1}^d |x_i|^{2\sigma_i}\Big)^2d\bx.
$$
Next, we make the change of variables $y_j=x_j^{2\sigma_j}$, $j=1,\ldots, d$ and
set $\mcD=\big\{\by\in\RR_+^d : \sum_{j=1}^d y_j<\prod_{i=1}^{d} y_i^{{\alpha_i}/{\sigma_i}}\big\}$.
We get
\begin{align*}
I_0(T)&\sim  \frac{T^{\frac{4\delta\bar\sigma+d}{2(1-\delta)\bar\sigma}}}{\pi^d\sigma_1\ldots \sigma_d}\int_\mcD \Big( \prod_{i=1}^{d}y_i^{\frac{\alpha_i}{\sigma_i}}
-\sum_{j=1}^d y_j\Big)^2 y_1^{\frac{1}{2\sigma_1}-1}\ldots y_d^{\frac{1}{2\sigma_d}-1} d\by.
\end{align*}
Now, we make another change of variables: $z_i={y_i}\big({\prod_{j=1}^{d}y_j^{{\alpha_j}/{\sigma_j}}}\big)^{-1}$.
Note that  $\prod_{i=1}^{d}z_i^{{\alpha_i}/{\sigma_i}}=\big(\prod_{i=1}^{d}y_i^{{\alpha_i}/{\sigma_i}}\big)^{1-\delta}$.
Therefore, using the notation $\Sigma_d=\big\{ \bz\in\RR_+^d : \|\bz\|_1 \leq 1\big\}$,
$$
I_0(T)\sim  \frac{T^{\frac{4\delta\bar\sigma+d}{2(1-\delta)\bar\sigma}}}{\pi^d\sigma_1\ldots \sigma_d}\int_{\Sigma_d}\Big(\prod_{i=1}^{d}z_i^{\frac{\alpha_i}{\sigma_i}}\Big)^{\frac{4\bar\sigma+d-2d\bar\sigma}{2\bar\sigma(1-\delta)}}(1-\|\bz\|_1)^2
z_1^{\frac{1}{2\sigma_1}-1}\ldots z_d^{\frac{1}{2\sigma_d}-1}\Delta(\bz)\; d\bz,
$$
where $\Delta(\bz)$ is the Jacobian. Standard algebra yields $\Delta(\bz)=\big(\prod_{i=1}^{d}z_i^{{\alpha_i}/{\sigma_i}}\big)^{d/(1-\delta)}/(1-\delta)$.
Next we give an explicit form for this integral $I_0(T)\sim {\pi^{-d}}{T^{\frac{4\delta\bar\sigma+d}{2(1-\delta)\bar\sigma}}} I$, 	
where
\begin{align*}
I&= \frac{1}{ \big(\prod_{i=1}^{d}\sigma_i\big)(1-\delta)}\int_{\Sigma_d}  \Big( \prod_{i=1}^{d}z_i^{\frac{\alpha_i}{\sigma_i}}\Big)^{\frac{4\bar\sigma+d}{2\bar\sigma(1-\delta)}}
(1-\|\bz\|_1)^2 z_1^{\frac{1}{2\sigma_1}-1}\ldots z_d^{\frac{1}{2\sigma_d}-1}d\bz.
\end{align*}
Now, the Liouville formula (see, for instance, \cite{ingster2011estimation}) combined with the well-known identity
$\int_0^1 u^{\alpha-1}(1-u)^{\beta-1}\,du=\Gamma(\alpha)\Gamma(\beta)/\Gamma(\alpha+\beta)$ yields
\begin{align*}
I&= \frac{\prod_{i=1}^{d}\Gamma\big(\frac{1}{2\sigma_i}+\frac{\alpha_i}{\sigma_i}\frac{4\bar\sigma+d}{2\bar\sigma(1-\delta)}\big)}
{\big(\prod_{i=1}^{d}\sigma_i\big)(1-\delta) \Gamma\big(\frac{d}{2\bar\sigma} +(2+\frac{d}{2\bar\sigma})\frac{\delta}{1-\delta}\big)}\int_0^1 (1-u)^2u^{\frac{d}{2\bar\sigma}+\frac{(4\bar\sigma+d)\delta}{2\bar\sigma(1-\delta)}-1}du\\
&= \frac{2\prod_{i=1}^{d}\Gamma(\kappa_i)}{\big(\prod_{i=1}^{d}\sigma_i\big)(1-\delta) \Gamma(\kappa+3)}
 = \frac{2\prod_{i=1}^{d}\Gamma(\kappa_i)}{\big(\prod_{i=1}^{d}\sigma_i\big)(1-\delta)(\kappa+2) \Gamma(\kappa+2)}
= \frac{2\pi^{d} C(d,\ssigma,\aalpha)}{\kappa+2}.
\end{align*}
Therefore,
$$
I_0(T)\sim \frac{2C(d,\ssigma,\aalpha)}{\kappa+2} T^{\frac{4\delta\bar\sigma+d}{2(1-\delta)\bar\sigma}}.
$$
Very similar computations imply that, as $T\to\infty$, we have
\begin{align*}
I_1(T)\sim\frac{T^{\frac{4\delta\bar\sigma+d}{2(1-\delta)\bar\sigma}}}{\pi^d}\frac{\prod_{i=1}^{d}\Gamma(\kappa_i)}{ \big(\prod_{i=1}^{d}\sigma_i\big)(1-\delta)\Gamma(\kappa+2)}= C(d,\ssigma,\aalpha) T^{\frac{4\delta\bar\sigma+d}{2(1-\delta)\bar\sigma}}.
\end{align*}

Note now that (\ref{eq:24}) is equivalent to $n^2T^2I_0(T)\sim 8T^4 (I_1(T)-I_0(T))^2z^2_{1-\gamma/2}$. Using the asymptotic equivalents for
$I_0$ and $I_1$ we have derived above, one directly checks that the value of $T_{n,\gamma}$ proposed in Proposition~\ref{premierex} satisfies (\ref{eq:24}).
Furthermore, since (\ref{eq:25}) is equivalent to $(r_{n,\gamma}^*)^2= I_1(T_{n,\gamma})/T_{n,\gamma}I_2(T_{n,\gamma})$, we get
$r_{n,\gamma}^*=C_\gamma^*r_n^*(1+o(1))$, as announced in proposition.

It remains to check that for the sequence $T_{n,\gamma}\asymp n^{\frac{4\bar\sigma(1-\delta)}{4\bar\sigma+d}}$ conditions
\textbf{[C1]}, \textbf{[C2]}, \textbf{[C5]}, \textbf{[C8]} and \textbf{[C9]} are fulfilled.
Using  the same method as the
one used above to evaluate $I_0$, we get
\begin{equation}\label{N(T)}
|\mcN(T_{n,\gamma})|\asymp n^{\frac{2d}{4\bar\sigma+d}}\quad \text{ and } \quad
M(T_{n,\gamma})=\sum_{\bl\in\mcN(T_{n,\gamma})} q_\bl^2\asymp n^{\frac{2(4\delta\bar\sigma+d)}{4\bar\sigma+d}}.
\end{equation}
The assumption $\bar\sigma>d/4$ implies $|\mcN(T_{n,\gamma})|\log|\mcN(T_{n,\gamma})|=o(n)$ and, as a
consequence, conditions \textbf{[C4]} and \textbf{[C8]} are true.
Furthermore, the second relation in (\ref{N(T)}) combined with $\delta<1$ implies \textbf{[C2]}. Condition \textbf{[C5]} follows
from the fact that all the nonzero entries of $\bq$ are lower-bounded by $1$.

In order to check \textbf{[C1]} and \textbf{[C9]}, we need to find an upper bound for
$\max_{l\in \mcN(T_{n,\gamma})} q_\bl$. In the following calculations, the term $C$ is
a constant which depends only on $d$, $\aalpha$ and $\ssigma$ and can vary from  line to line.
Let $l \in \NL(T)$, then  $c_\bl\leq Tq_\bl$, which  implies, for every $i=1,\ldots,d$,
$l_i^{2(\sigma_i-\alpha_i)}\leq C T\prod_{j\neq i}\nolimits l_j^{2\alpha_j}$.
In particular
\begin{equation}\label{l1}
l_1^{2(\sigma_1-\alpha_1)}\leq C T\prod_{j\neq 1}l_j^{2\alpha_j},
\quad l_2^{2(\sigma_2-\alpha_2)}\leq C T\prod_{j\neq 2}l_j^{2\alpha_j},
\quad l_3^{2(\sigma_3-\alpha_3)}\leq C T\prod_{j\neq 3}l_j^{2\alpha_j}.
\end{equation}
Injecting the first inequality of (\ref{l1}) in the second one, we obtain
$$
l_2^{2(\sigma_2-\alpha_2)}
    \leq C T\Big(\prod_{j\ge 3}l_j^{2\alpha_j}\Big)\Big(T\prod_{j\ge 2}l_j^{2\alpha_j}\Big)^{\frac{\alpha_1}{\sigma_1-\alpha_1}}
    \leq C \Big(T\prod_{j\ge 3}l_j^{2\alpha_j}\Big)^{\frac{\sigma_1}{\sigma_1-\alpha_1}}\Big(l_2^{2\alpha_2}\Big)^{\frac{\alpha_1}{\sigma_1-\alpha_1}}.
$$
Hence
\begin{equation}\label{ll2}
l_2^{2\sigma_2}\leq  C \Big(T\prod_{j\geq 3}\nolimits l_j^{2\alpha_j}\Big)^{\frac{1}{1-\alpha_1/\sigma_1-\alpha_2/\sigma_2}}
\end{equation} and by symmetry,
\begin{equation}\label{ll1}
l_1^{2\sigma_1}\leq  C \Big(T\prod_{j\geq 3}\nolimits l_j^{2\alpha_j}\Big)^{\frac{1}{1-\alpha_1/\sigma_1-\alpha_2/\sigma_2}}.
\end{equation}
Next, using (\ref{ll2}), (\ref{ll1}) and the third inequality in (\ref{l1}), we get
$$
l_3^{2\sigma_3}\leq  C \Big(T\prod_{j\geq 4}\nolimits l_j^{2\alpha_j}\Big)^{\frac{1}{1-\alpha_1/\sigma_1-\alpha_2/\sigma_2-\alpha_3/\sigma_3}}.$$
Iterations of the previous process lead to the inequality $\max_j l_j^{2\sigma_j}\leq  C T^{{1}/{(1-\delta)}}$. Therefore,
$\max_{\bl\in\mcN(T)} q_\bl= C \prod_{j=1}^d\nolimits l_j^{2\alpha_j}\leq C T^{\frac{\delta}{1-\delta}}$.
Combining this bound with $T_{n,\gamma}\asymp n^{\frac{4\bar\sigma(1-\delta)}{4\bar\sigma+d}}$  and  (\ref{N(T)}) yields the inequalities of
{\bf [C1]} and {\bf [C9]}.

\section{Proof of Proposition \ref{deuxiemeex}}
As in the previous subsection, we begin with the calculation of $I_0$.
Setting  $x_{\bl,i}=\frac{2\pi l_i}{T^{\frac{1}{\sigma-1}}} $  and using the same method to get an integral, we have
\begin{align*}
I_0&=T^{\frac{4}{\sigma-1}}\sum_{\bl\in\ZZ^d} \Big[ \|\bx_{\bl}\|_2^2-(\bbeta^\top\bx_{\bl})^2-\|\bx_{\bl}\|_{2\sigma}^{2\sigma}\Big]_+^2\\
&\sim \frac{T^{\frac{d+4}{\sigma-1}}}{(2\pi)^d}\int_{\RR^d}
\Big[ \|\bx\|_2^2-(\bbeta^\top\bx)^2-\|\bx\|_{2\sigma}^{2\sigma}\Big]_+^2d\bx.
\end{align*}
This implies the asymptotic relation $I_0\sim  C_0T^{\frac{d+4}{\sigma-1}}$  with  the constant
$C_0=\frac{1}{(2\pi)^d}\int_{\RR^d}
\big[ \sum_{i=1}^d (x_i-\frac{1}{d}\sum_{j=1}^d x_j)^2-\sum_{i=1}^d x_i^{2\sigma}\big]_+^2d\bx$.
Similar computations yield $I_1\sim C_1T^{(d+4)/(\sigma-1)}$ and $I_2\sim  C_2T^{(d+4)/(\sigma-1)}$,
where $C_1$ and $C_2$ have the values given in the paragraph preceding  the proposition.

The rest of the proof can be carried out exactly in the same way as the proof of the previous proposition, based
on the relation $N(T)\asymp T^{\frac{d}{\sigma-1}}$ and $M(T)\asymp T^{\frac{d+4}{\sigma-1}}$.

\section{Proofs of results stated in Section~\ref{sec:NPF}}

\subsection{Proof of Theorem \ref{nonpositiveupper}}

The arguments are almost the same as in the proof of Theorem~\ref{thm_1}. We use the array $\bw_n$ with entries
$w_l=q_l\1_{\{l\in\mcN(T)\}}/M(T)^{1/2}$ and the kernel
$G_n(\bt_1,\bt_2)=\sum_{l\in \mcL} w_l\varphi_l(\bt_1)\varphi_l(\bt_2)$
in order to define the linear U-test statistic:
$$
U_n=\binom{n}{2}^{-1/2} \sum_{1\le i<j\le n} x_ix_j G_n(\bt_i,\bt_j).
$$
We write as $U_n=U_{n,0}+U_{n,1}+U_{n,2}$, where
\begin{align*}
U_{n,0}&= \binom{n}{2}^{-1/2} \sum_{i<j} \xi_i\xi_j G_n(\bt_i,\bt_j),\quad
U_{n,1}= \binom{n}{2}^{-1/2} \sum_{i<j} (\xi_if(\bt_j)+\xi_j f(\bt_i)) G_n(\bt_i,\bt_j)
\end{align*}
and $U_{n,2}= \binom{n}{2}^{-1/2} \sum_{i<j} f(\bt_i)f(\bt_j) G_n(\bt_i,\bt_j)$. The first and
the second moments of this U-statistic are described in the next result, in which we use the notation
$\ttT_\bw[f]=\sum_l w_l \theta_l[f]\varphi_l$.

\begin{lemma}\label{lem:AD2}
Let  $\bw_n=(w_{l,n})_{l\in\mcL}$ be an array containing only a finite number of nonzero entries and such that
$\sum_{l\in\mcL} w_{l,n}^2=1$. Let $\mcL(\bw_n)$ be the support of $\bw_n$. The expectation of the U-statistic $U_{n}$ is given by:
\begin{align*}
E_f[U_{n}]=E_f[U_{n,2}]=\bar h_n[f,\bw_n]= \Big(\frac{n(n-1)}{2}\Big)^{1/2}\sum_l w_{l}\theta_l^2[f].
\end{align*}
Furthermore, if {\bf [D2]} holds true, then $E[U_{n,0}^2]=1$, $E[U_{n,1}^2]\le 2D_2\|\bw_n\|_\infty^2\|\bw_n\|_0\|f\|_2^2$ and
\begin{align*}
\text{\rm Var}[U_{n,2}]
	&\leq D_2 \|\bw_n\|_\infty^2 \|\bw_n\|_0\|f\|_4^4+\frac{2n}{3}\|f\cdot \ttT_\bw[f]\|_2^2.
\end{align*}
\end{lemma}
\begin{proof}
This result can be proved along the lines of the proof of Lemma~\ref{lem:AD1}. The only difference is in the evaluation
of the term $A_{n,2}$, for which we have
\begin{align*}
A_{n,2}&=\frac{4}{n(n-1)}\binom{n}{3} \sum_{l,l'\in \mcL(\bw_n)}w_{l}w_{l'}\theta_l[f]\theta_{l'}[f]
    	   \Big\{\int f(\bt)^2 \varphi_l(\bt)\varphi_{l'}(\bt)\,d\bt\Big\}\\
       &=\frac{2(n-2)}{3}\Big\{\int f(\bt)^2\Big( \sum_l w_{l}\theta_l[f]\varphi_l(\bt)\Big)^2\,d\bt\Big\}
       \le \frac{2n}{3}\|f\cdot \ttT_\bw[f]\|_2^2.
\end{align*}
This yields the desired result.
\end{proof}

Let us now study the type I and type II error probabilities of the test $\phi_n(T)=\1_{\{|U_n(T)|> u\}}$.
\paragraph{Evaluation of type I error}
Using Tchebychev's inequality, for every $u>|E[U_n(T)]|$, we have
\begin{align*}
\sup_{f\in\mcF_0}P_f\big(|U_n(T)|>u\big)
    &\leq \sup_{f\in\mcF_0}P_f\Big(\big|U_n(T)-E[U_n(T)]\big|>u-\big|E[U_n(T)]\big|\Big)\\
    &\leq \sup_{f\in\mcF_0} \frac{\text{Var}(U_n(T))}{(u-|E[U_n(T)]|)^2}.
\end{align*}
Let us denote $\nu_{n,T}= nT^{-1}(2M(T))^{-1/2}$. Using Lemma~\ref{lem:AD2}, we get
\begin{align*}
|E[U_n(T)]|&\le \frac{n}{\sqrt{2M(T)}}\Big|\sum_{l\in\mcN(T)} q_l\theta_l^2[f]\Big|=T\nu_{n,T}\Big|\sum_{l\in\mcN(T)} q_l\theta_l^2[f]\Big|.
\end{align*}
Since, under $H_0$, we have $Q[f]=\sum_l q_l\theta_l[f]^2=0$ and $\sum_l c_l\theta_l[f]^2\le 1$, the last sum can be bounded as follows:
$\big|\sum_{l\in\mcN(T)} q_l\theta_l^2[f]\big|=\big|\sum_{l:|q_l|<c_l/T} q_l\theta_l^2[f]\big|\le T^{-1}\sum_l c_l\theta_l[f]^2\le T^{-1}$.
Thus, $|E[U_n(T)]|=|\bar h_n[\bw_n,f]|\le \nu_{n,T}$. Combining this bound with those of Lemma~\ref{lem:AD2}, we arrive at
\begin{align*}
\sup_{f\in\mcF_0}P_f(\phi_n(T)=1)
    &\leq \frac{3(1+2D_1 D_2 D_3^2+D_1 D_2 D_3^4+2n D_4/(3M(T))}{(u-\nu_{n,T})^2}\\
    & =   \frac{B_1+B_2nM(T)^{-1}}{2(u-\nu_{n,T})^2}
\end{align*}
Consequently, if we choose $u\ge \nu_{n,T}+\big(B_1+B_2 nM(T)^{-1}\big)^{1/2}\gamma^{-1/2}$, then
$\sup_{f\in\mcF_0}P_f(\phi_n(T)=1) \leq  \frac{\gamma}{2}$.

\paragraph{Evaluation of  type II error} Using similar arguments, we get
\begin{align*}
\sup_{f\in\mcF_1(\rho)} P_f(\phi_n(T)=0)&=\sup_{f\in\mcF_1(\rho)}P_f(|U_n(T)|\leq u)\\
&\le \sup_{f\in\mcF_1(\rho)}P_f\big(|E[U_n(T)]|-|U_n(T)-E[U_n(T)]|\leq u\big)\\
&\le \sup_{f\in\mcF_1(\rho)}P_f\big(T\nu_{n,T}|Q[f]|-\nu_{n,T}-\big|U_n(T)-E[U_n(T)]\big|\leq u\big)\\
&\le P\big(T\nu_{n,T}\rho^2-\big|U_n(T)-E[U_n(T)]\big|\leq u+\nu_{n,T}\big).
\end{align*}
This can also be written as:
\begin{align*}
\sup_{f\in\mcF_1(\rho)} P_f(\phi_n(T)=0)
\le P\Big(\big|U_n(T)-E[U_n(T)]\big|\geq (T\rho^2-1)\nu_{n,T}-u\Big).
\end{align*}
Using the Tchebychev inequality and the evaluations obtained in Lemma~\ref{lem:AD2}, we get
\begin{align*}
\sup_{f\in\mcF_1(\rho)} P_f(\phi_n(T)=0)
\le \frac{B_1+B_2nM(T)^{-1}}{2\big((T\rho^2-1)\nu_{n,T}-u\big)^2}.
\end{align*}
Clearly, the right hand-side of this inequality is lower than $\gamma/2$ if
$$
\rho^2\ge \bigg[u+\frac1{\gamma^{1/2}}\bigg(B_1+\frac{B_2n}{M(T)}\bigg)^{1/2}\bigg]\frac{{1}}{T\nu_{n,T}}+\frac{1}{T}.
$$
This completes the proof of Theorem~\ref{nonpositiveupper}.

\subsection{Proof of Corollary~\ref{cor}}

It is enough to remark that (since $M(\cdot)$ is increasing and $T_n\le T_n^0$)
$$
\frac{\sqrt{M(T_n)}}{n}\le \frac{\sqrt{M(T_n^0)}}{n} \asymp \frac{1}{T_n^0}\le \frac1{T_n}
$$
and $\frac1{\sqrt n}\le T_n^{-1}$. In view of these inequalities, the claim of the corollary
immediately follows from Theorem~\ref{nonpositiveupper}.

\subsection{Proof of Theorem \ref{nonpositivelower}}

We start by proving that the minimax rate of separation is lower bounded by $n^{-1/4}$.
Let $l_a=\text{argmin}_{l\in\mcL_a}\{c_l\}$ for $a\in\{+,-\}$. We define two functions
$f_0$ and $f_1$ as linear combinations of the basis functions $\varphi_{l_-}$ and $\varphi_{l_+}$. More precisely,
$f_i=\theta_{i,-}\varphi_{l_-}+\theta_{i,+}\varphi_{l_+}$, for $i=0,1,$
with
$$
\theta_{0,-}^2=\frac{|q_{l_+}|}{c_{l_-}|q_{l_+}|+c_{l_+}|q_{l_-}|} \quad \theta_{0,+}^2=\frac{|q_{l_-}|}{c_{l_-}q_{l_+}+c_{l_+}|q_{l_-}|}
$$
and, for some $z>0$,
$$
\theta_{1,-}=\theta_{0,-} \quad \theta_{1,+}^2=\theta_{0,+}^2-z/\sqrt{n}.
$$
One easily checks that $f_0\in \mcF_0$ and $f_1\in\mcF_1(r_n)$ with $r_n^2=zq_{l_+}/\sqrt{n}$.
Furthermore, the Kullback-Leibler divergence $K(P_{f_0}, P_{f_1})= \int \log \frac{dP_{f_0}}{dP_{f_1}} dP_{f_0}$
between the probability measures $P_{f_0}$ and $P_{f_1}$ can be bounded as follows:
\begin{align*}
K(P_{f_0}, P_{f_1})
    &= E\bigg( E_{f_0}\bigg[ \log \frac{dP_{f_0}}{dP_{f_1}}(x_1,\ldots,x_n,\bt_1,\ldots,\bt_n)\bigg|\bt_1,\ldots,\bt_n\bigg]\bigg)\\
    &= E\bigg( E_{f_0}\bigg[ \sum_{i=1}^n (x_i-f_1(\bt_i))^2-(x_i-f_0(\bt_i))^2\bigg|\bt_1,\ldots,\bt_n\bigg]\bigg)\\
    &=nE\big[\big( f_0(\bt_1)-f_1(\bt_1)\big)^2\big]=n\sum_{a\in\{+,-\}}\nolimits(\theta_{0,a}-\theta_{1,a})^2\\
    &=n\big(\theta_{0,+}-|{\theta_{0,+}^2-zn^{-1/2}}|^{1/2}\big)^2\le z^2(2\theta_{0,+})^{-2}.
\end{align*}
To conclude, it suffices to use inequality (2.74) from \citep{tsybakov2009introduction}, which implies that
$\gamma_n(\mcF_0,\mcF_1(r_n))\ge 0.25 e^{-z^2(2\theta_{0,+})^{-2}}=\gamma$ for $z=2\theta_{0,+}[\ln(4\gamma)^{-1}]^{1/2}$.

It remains to prove the second assertion of the theorem. To ease notation, we write $T_n$ instead of $T_n^0$ and set
$$
Q_a[f]=\sum_{l\in {\mcL_a}}\nolimits q_l\theta_l^2[f]\qquad \text{ and } \qquad
\mcF_a=\{ f : Q_a[f]=0\},\qquad \text{for}\qquad a\in\{+,-\}.
$$
Let us assume that $M_+(T_n)\geq M_-(T_n)$.
We use the fact that testing $Q[f]=0$ against $|Q[f]|\geq r_n^2$, with $f\in \Sigma$ is harder than testing $Q_+[f]=0$ against
$Q_+[f]\geq r_n^2$, with $f\in \mcF_-$.

The rest of the  proof follows the same steps as those of the proof of Theorem \ref{thm_2}.
As indicated in  Remark \ref{tauxseulementlower}, we use as $\pi_n$ the simplified prior
for which $\theta_l$'s are independent Gaussian random variables with zero mean and variance
$a_l=\frac{q_l}{2T_nM_+(T_n)}\1_{\{l\in {\mcL_+}\cup \mathcal{N}(T_n)\}}$. It is an easy exercice
to show that conditions \textbf{[L1]}-\textbf{[L5]} of Proposition~\ref{prop_4} are fulfilled with
$\delta=1/2$. This completes the proof of the theorem.

\bibliography{deuxiemearticle_derniere_version}

\end{document}